\documentclass[12pt]{amsart}
\usepackage{amsmath}
\usepackage{amstext}
\usepackage{amssymb}
\usepackage{amsthm}
\swapnumbers
\theoremstyle{plain}
\newtheorem{thm}{Theorem}[section]
\newtheorem{lem}[thm]{Lemma}
\newtheorem{prop}[thm]{Proposition}
\newtheorem{cor}[thm]{Corollary}

\theoremstyle{definition}
\newtheorem{defi}[thm]{Definition}
\newtheorem{defs}[thm]{Definitions}

\newtheorem{ntn}[thm]{Notation}

\newtheorem{rmds}[thm]{Reminders}
\theoremstyle{remark}
\newtheorem*{note}{Note}
\newtheorem{rmk}[thm]{Remark}
\newtheorem{rmks}[thm]{Remarks}

\DeclareMathOperator{\height}{ht} 
 
\DeclareMathOperator{\ext}{exten} \DeclareMathOperator{\Hom}{Hom}

  \DeclareMathOperator{\Reg}{Reg}

\DeclareMathOperator{\Spec}{Spec} 
 \DeclareMathOperator{\ann}{ann}
\DeclareMathOperator{\grann}{gr-ann} 
\def\Z{\mathbb Z}
\def\N{\mathbb N}

\def\fa{{\mathfrak{a}}}

\def\fb{{\mathfrak{b}}}
\def\fB{{\mathfrak{B}}}
\def\fc{{\mathfrak{c}}}

\def\fd{{\mathfrak{d}}}
\def\fk{{\mathfrak{k}}}

\def\fm{{\mathfrak{m}}}

\def\fn{{\mathfrak{n}}}

\def\fp{{\mathfrak{p}}}

\def\fq{{\mathfrak{q}}}

\def\nn{\relax\ifmmode{\mathbb N_{0}}\else$\mathbb N_{0}$\fi}
\def\lra{\longrightarrow}

\begin{document}

\title[Big test elements for excellent $F$-pure rings]{An excellent $F$-pure
ring of prime characteristic has a big tight closure test element}
\author{RODNEY Y. SHARP}
\address{Department of Pure Mathematics,
University of Sheffield, Hicks Building, Sheffield S3 7RH, United
Kingdom} \email{R.Y.Sharp@sheffield.ac.uk}


\date{\today}


\begin{abstract}
In two recent papers, the author has developed a theory of graded
annihilators of left modules over the Frobenius skew polynomial ring
over a commutative Noetherian ring $R$ of prime characteristic $p$,
and has shown that this theory is relevant to the theory of test
elements in tight closure theory. One result of that work was that,
if $R$ is local and the $R$-module structure on the injective
envelope $E$ of the simple $R$-module can be extended to a structure
as a torsion-free left module over the Frobenius skew polynomial
ring, then $R$ is $F$-pure and has a tight closure test element. One
of the central results of this paper is the converse, namely that,
if $R$ is $F$-pure, then $E$ has a structure as a torsion-free left
module over the Frobenius skew polynomial ring; a corollary is that
every $F$-pure local ring of prime characteristic, even if it is not
excellent, has a tight closure test element. These results are then
used, along with embedding theorems for modules over the Frobenius
skew polynomial ring, to show that every excellent (not necessarily
local) $F$-pure ring of characteristic $p$ must have a so-called
`big' test element.
\end{abstract}

\maketitle

\setcounter{section}{-1}
\section{\sc Introduction}
\label{intro}

This paper is concerned with the existence of tight closure test
elements for an $F$-pure (commutative Noetherian) ring of prime
characteristic.

In two recent papers \cite{ga} and \cite{gatcti}, the author has
developed a theory of graded annihilators of left modules over the
Frobenius skew polynomial ring (over a commutative Noetherian ring
of prime characteristic), and shown that this theory is relevant to
the theory of tight closure test elements. We shall therefore
recall, in this introduction, certain definitions from M. Hochster's
and C. Huneke's theory of tight closure, but they will be presented
below in terms of the Frobenius skew polynomial ring.

Throughout the paper, $R$ will denote a commutative Noetherian ring
of prime characteristic $p$. We shall only assume that $R$ is local
when this is explicitly stated; then, the notation `$(R,\fm)$' will
denote that $\fm$ is the maximal ideal of $R$. We shall always
denote by $f:R\lra R$ the Frobenius homomorphism, for which $f(r) =
r^p$ for all $r \in R$. The {\em Frobenius skew polynomial ring over
$R$\/} is the skew polynomial ring $R[x,f]$ associated to $R$ and
$f$ in the indeterminate $x$ over $R$. Recall that $R[x,f]$ is, as a
left $R$-module, freely generated by $(x^i)_{i \in \nn}$ (throughout
the paper, we use $\N$ and $\nn$ to denote the set of positive
integers and the set of non-negative integers, respectively),
 and so consists
 of all polynomials $\sum_{i = 0}^n r_i x^i$, where  $n \in \nn$
 and  $r_0,\ldots,r_n \in R$; however, its multiplication is subject to the
 rule
 $
  xr = f(r)x = r^px$ for all $r \in R.$
Note that $R[x,f]$ can be considered as a positively-graded ring
$R[x,f] = \bigoplus_{n=0}^{\infty} R[x,f]_n$, with $R[x,f]_n = Rx^n$
for all $n \in \nn$.

If, for $n \in \N$, we endow $Rx^n$ with its natural structure as an
$(R,R)$-bimodule (inherited from its being a graded component of
$R[x,f]$), then $Rx^n$ is isomorphic (as $(R,R)$-bimodule) to $R$
viewed as a left $R$-module in the natural way and as a right
$R$-module via $f^n$, the $n$th iterate of the Frobenius ring
homomorphism. We can now write that $R$ is {\em $F$-pure\/}
precisely when, for each $R$-module $N$, the map $\psi_N : N \lra
Rx\otimes_RN$ for which $\psi_N(g) = x\otimes g$ for all $g \in N$
is injective.

A left $R[x,f]$-module $H$ is said to be {\em $x$-torsion-free\/}
precisely when $$\Gamma_x(H) := \left\{h \in H : x^nh = 0 \mbox{~for
some~} n \in \N\right\} = 0.$$ It is worth pointing out right away
that, when $R$ is $F$-pure, the left $R[x,f]$-module
$R[x,f]\otimes_RN$ is $x$-torsion-free, for each $R$-module $N$.

We now formulate some of the basic definitions of tight closure
theory in terms of the Frobenius skew polynomial ring. The theory of
tight closure was invented by M. Hochster and C. Huneke
\cite{HocHun90}, and many applications have been found for the
theory: see \cite{Hunek96}.

We use $R^{\circ}$ to denote the complement in $R$ of the union of
the minimal prime ideals of $R$. Let $L$ and $M$ be $R$-modules and
let $K$ be a submodule of $L$. Observe that there is a natural
structure as $\nn$-graded left $R[x,f]$-module on $R[x,f]\otimes_RM
= \bigoplus_{n\in\nn}(Rx^n\otimes_RM)$. An element $m \in M$ belongs
to $0^*_M$, the {\em tight closure of the zero submodule in $M$\/},
if and only if there exists $c \in R^{\circ}$ such that the element
$1 \otimes m \in (R[x,f]\otimes_RM)_0$ is annihilated by $cx^j$ for
all $j \gg 0$: see Hochster--Huneke \cite[\S 8]{HocHun90}.
(Incidentally, in discussions of this type, we shall often tacitly
identify $$(R[x,f]\otimes_RM)_0 = Rx^0\otimes_RM = R\otimes_RM$$
with $M$ in the obvious way.) Furthermore, the tight closure $K^*_L$
of $K$ in $L$ is the inverse image, under the natural epimorphism $L
\lra L/K$, of $0^*_{L/K}$, the tight closure of $0$ in $L/K$. In
general, we have $K \subseteq K^*_L$; we say that $K$ is {\em
tightly closed in $L$\/} precisely when $K = K^*_L$, that is, if and
only if $0^*_{L/K} = 0$.

The ring $R$ is said to be {\em weakly $F$-regular\/} precisely when
every ideal of $R$ is tightly closed in $R$, and to be {\em
$F$-regular\/} if and only if each ring of fractions $S^{-1}R$ of
$R$ with respect to an arbitrary multiplicatively closed subset $S$
of $R$ is weakly $F$-regular.

A {\em test element for modules\/} for $R$ is an element $c \in
R^{\circ}$ such that, for every finitely generated $R$-module $M$
and every $j \in \nn$, the element $cx^j$ annihilates $1 \otimes m
\in (R[x,f]\otimes_RM)_0$ for every $m \in 0^*_M$. The phrase `for
modules' is inserted because Hochster and Huneke have also
considered a concept of a {\em test element for ideals} for $R$,
which is defined to be an element $c \in R^{\circ}$ such that, for
every {\em cyclic\/} $R$-module $M$ and every $j \in \nn$, the
element $cx^j$ annihilates $1 \otimes m \in (R[x,f]\otimes_RM)_0$
for every $m \in 0^*_M$. When $R$ is reduced and excellent, the
concepts of test element for modules and test element for ideals for
$R$ coincide: see \cite[Discussion (8.6) and Proposition
(8.15)]{HocHun90}. With this, and the fact that the main existence
results for test elements in this paper only apply to excellent
rings that are $F$-pure (and therefore reduced), in mind, I shall
use the phrase `{\em test element for $R$}' as an abbreviation for
`test element for modules for $R$'.

A {\em big test element\/} for $R$ is defined to be an element $c
\in R^{\circ}$ such that, for {\em every\/} $R$-module $M$ and every
$j \in \nn$, the element $cx^j$ annihilates $1 \otimes m \in
(R[x,f]\otimes_RM)_0$ for every $m \in 0^*_M$.

It is generally accepted that, currently, the best known results
about the existence of test elements are those of Hochster--Huneke
in \cite[\S 6]{HocHun94}.  Here is one of their results. (For $c\in
R$, we denote by $R_c$ the ring of fractions of $R$ with respect to
the multiplicatively closed set consisting of the powers of $c$.)

\begin{thm}[M. Hochster and C. Huneke {\cite[Theorem
(6.1)(b)]{HocHun94}}] \label{in.1} A reduced algebra $R$ of finite
type over an excellent local ring of characteristic $p$ has a test
element.

In fact, if $c \in R^{\circ}$ is such that $R_c$ is regular, then
some power of $c$ is a test element for $R$.
\end{thm}

However, it is also generally accepted that the proof of the above
theorem in \cite{HocHun94} is quite difficult, because it depends on
the so-called `$\Gamma$-construction'.

One of the main results of \cite{gatcti} is that, if $(R,\fm)$ is
local and the $R$-module structure on the injective envelope
$E_R(R/\fm)$ of $R/\fm$ can be extended to an $x$-torsion-free left
$R[x,f]$-module structure, then $R$ is $F$-pure and has a test
element for modules. In this paper we shall prove the converse, and
so obtain (in Theorem \ref{fp.1} below) the following
characterization of $F$-pure local rings.

\vspace{0.1in}

\noindent{\bf Theorem.} {\it Suppose that $(R,\fm)$ is local. Then
$R$ is $F$-pure if and only if the $R$-module structure on
$E_R(R/\fm)$ can be extended to an $x$-torsion-free left
$R[x,f]$-module structure.}

\vspace{0.1in}

A corollary is that every $F$-pure local ring (such a ring must be
reduced) has a test element for modules, even if it is not
excellent.

The methods from \cite{ga} and \cite{gatcti} based on graded
annihilators are particularly well-suited to $x$-torsion-free left
$R[x,f]$-modules. A key point is that those methods associate, with
an $x$-torsion-free left $R[x,f]$-module $E$ that is Artinian as an
$R$-module, a certain finite set ${\mathcal I}(E)$ of radical ideals
of $R$, and, in some circumstances (such as the case where $(R,\fm)$
is local and $E$ is $R$-isomorphic to the injective envelope of the
simple $R$-module), this set ${\mathcal I}(E)$ has connections with
tight closure test elements.

This paper uses those ideas, in conjunction with embedding theorems
for modules over the Frobenius skew polynomial ring, to prove (in
Theorem \ref{fp.32} below) the following existence theorem.

\vspace{0.1in}

\noindent{\bf Theorem.} {\it Suppose that $R$ is $F$-pure and
excellent (but not necessarily local). Then $R$ has a big test
element.

In fact, any $c \in R^{\circ}$ for which $R_c$ is regular must be a
big test element for $R$.}

\vspace{0.1in}

This theorem only applies to $F$-pure excellent rings, and so has
rather limited applicability compared with the Hochster--Huneke
Theorem \ref{in.1}; on the other hand, it does apply to all $F$-pure
excellent rings and not just to those $F$-pure rings that are
algebras of finite type over an excellent local ring. Also, the
theorem asserts the existence of big test elements, as opposed to
test elements. In fact, one can deduce fairly quickly from known
results of Hochster and Huneke that, if $R$ is $F$-pure and
excellent and $c \in R^{\circ}$ is such that $R_c$ is regular, then
$c$ is a test element for $R$. We prove this claim now. However, the
reader should note that the main aim of this paper is to show that
methods based on graded annihilators from \cite{ga} and
\cite{gatcti} provide a completely new approach to questions about
the existence of tight closure test elements.

\begin{thm}
\label{in.2} Suppose that $R$ is $F$-pure and excellent (but not
necessarily local). Then $R$ has a test element.

In fact, any $c \in R^{\circ}$ for which $R_c$ is regular must be a
test element for $R$.
\end{thm}

\begin{proof} Since $F$-pure rings are reduced, the concepts of `test
element for modules' and `test element for ideals' coincide for $R$:
see Hochster--Huneke \cite[Discussion (8.6) and Proposition
(8.15)]{HocHun90}. By Hochster--Huneke \cite[Proposition
6.1(a)]{HocHun90}, it is enough, in order to show that $c$ is a test
element for $R$, to show that $c/1 \in R_{\fp}$ is a test element
for $R_{\fp}$ for each prime ideal $\fp$ of $R$.

By Hochster and J. L. Roberts \cite[Lemma 6.2]{HocRob74}, the
localization $R_{\fp}$ is again $F$-pure; of course, $R_{\fp}$ is
excellent. Furthermore, since $\left(R_{\fp}\right)_{c/1}$, the ring
of fractions of $R_{\fp}$ with respect to the set of powers of the
element $c/1$ of $(R_{\fp})^{\circ}$, is a ring of fractions of
$R_c$, it is regular. It is therefore enough for us to prove the
claim in the case where $(R,\fm)$ is local.

In that case, we can appeal to the Hochster--Huneke Theorem
\ref{in.1} to deduce the existence of $e \in \nn$ such that
$c^{p^e}$ is a test element for $R$. We shall now deduce that $c$
itself is a test element for $R$.

Let $M$ be a finitely generated $R$-module, and let $m \in 0^*_M$.
Therefore $c^{p^e}x^j(1 \otimes m) = 0$ in $R[x,f]\otimes_RM$ for
all $j \in \nn$. Thus
$$
x^ecx^i(1 \otimes m) = c^{p^e}x^{e+i}(1 \otimes m) = 0 \quad
\mbox{for all~} i \in \nn.
$$
However, $R[x,f]\otimes_RM$ is $x$-torsion-free because $R$ is
$F$-pure, and so $cx^i(1 \otimes m) = 0$ for all $i \in \nn$. This
is true for all $m \in 0^*_M$, for each finitely generated
$R$-module $M$. Therefore $c$ is a test element for $R$.
\end{proof}

The main aim of this paper is to show that the theory of graded
annihilators can be used to strengthen the first paragraph of
Theorem \ref{in.2} by the insertion of `big' just before `test
element', and thereby provide an alternative proof that is
independent of the Hochster--Huneke Theorem \ref{in.1} and the
$\Gamma$-construction.

It is a pleasure to record my gratitude to Mordechai Katzman for
many discussions about matters related to the work in this paper,
and particularly about the material in \S 2 below.

\begin{note} The referee has asked me to report that results of
Hochster and Huneke on test elements have been extended by H.
Elitzur in his unpublished 2003 University of Michigan thesis `Tight
closure in Artinian modules'. As Elitzur's work in this context has
not been published, I rely completely on the referee for the
accuracy of the comments in this note.

Elitzur defines a `general test element for $R$' to be an element $c
\in R^{\circ}$ such that, for every $R$-module $L$ and every
submodule $K$ of $L$, an element $m$ of $L$ belongs to $K^*_L$ if
and only if $cx^j$ annihilates $1 \otimes (m+K) \in
(R[x,f]\otimes_R(L/K))_0$ for all $j \gg 0$. The referee reports
that Elitzur's results about general test elements are stated for
$F$-finite rings, and that Proposition 3.4 of Elitzur's thesis uses
Hochster's and Huneke's \cite[Theorem 5.10]{HocHun94} to produce
general test elements.
\end{note}

\section{Some notation and known results}
\label{pl}

This paper builds on the results of \cite{ga} and \cite{gatcti}, and
we shall make much use of notation, terminology and results from
\cite[\S 1]{ga} and \cite[\S 1]{gatcti}.

\begin{ntn}
\label{nt.1z} Let $H$ be a left $R[x,f]$-module. The {\em graded
annihilator $\grann_{R[x,f]}H$ of $H$\/} is defined in
\cite[1.5]{ga} and is the largest graded two-sided ideal of $R[x,f]$
that annihilates $H$. We shall use $\mathcal{G}(H)$ (or
$\mathcal{G}_{R[x,f]}(H)$ when it is desirable to emphasize the ring
$R$) to denote the set of all graded annihilators of
$R[x,f]$-submodules of $H$.

Recall from \cite[1.5]{ga} that an $R[x,f]$-submodule of $H$ is said
to be a {\em special annihilator submodule of $H$\/} if it has the
form $\ann_H(\fB)$ for some {\em graded\/} two-sided ideal $\fB$ of
$R[x,f]$. As in \cite{ga}, we shall use $\mathcal{A}(H)$ to denote
the set of special annihilator submodules of $H$, although we shall
occasionally expand this notation to $\mathcal{A}_{R[x,f]}(H)$ when
it is desirable to specify $R$.
\end{ntn}

A basic technique from \cite{ga} is provided by the following.

\begin{lem} [{\cite[Lemma 1.7(v)]{ga}}] \label{ga1.7(v)} Let $H$ be
left $R[x,f]$-module. Then the map $\Gamma : \mathcal{A}(H) \lra
\mathcal{G}(H)$ given by
$$
\Gamma(N) = \grann_{R[x,f]}N \quad \mbox{for all~} N \in
\mathcal{A}(H)
$$ is an order-reversing bijection. Its inverse\/ $\Gamma^{-1} : \mathcal{G}(H) \lra
\mathcal{A}(H)$, also order-reversing, is given by
$$
\Gamma^{-1}(\fB) = \ann_{H}\fB \quad \mbox{for all~} \fB \in
\mathcal{G}(H).
$$
\end{lem}

When the left $R[x,f]$-module $H$ is $x$-torsion-free,
$\grann_{R[x,f]}H$ has the form $$\fb R[x,f] = \bigoplus _{n\in\nn}
\fb x^n$$ for some radical ideal $\fb$ of $R$ (by \cite[Lemma
1.9]{ga}), and, in that case, we write $\mathcal{I}(H)$ for the set
of (necessarily radical) ideals $\fc$ of $R$ for which there is an
$R[x,f]$-submodule $N$ of $H$ such that $\grann_{R[x,f]}N = \fc
R[x,f]$; in these circumstances, the members of $\mathcal{I}(H)$ are
referred to as the {\em $H$-special $R$-ideals}, and we have
$\mathcal{G}_{R[x,f]}(H) = \left\{\fb R[x,f] : \fb \in
\mathcal{I}(H)\right\}$. Lemma \ref{ga1.7(v)}, in this special case,
leads to the following.

\begin{prop} [{\cite[Proposition 1.11]{ga}}] \label{ga1.11} Let $G$
be an $x$-torsion-free left $R[x,f]$-module.

There is an order-reversing bijection $\Delta : \mathcal{A}(G) \lra
\mathcal{I}(G)$ given by
$$
\Delta (N) = \left(\grann_{R[x,f]}N\right)\cap R = (0:_RN) \quad
\mbox{for all~} N \in \mathcal{A}(G).
$$
The inverse $\Delta^{-1} : \mathcal{I}(G) \lra \mathcal{A}(G),$ also
order-reversing, is given by
$$
\Delta^{-1}(\fb) = \ann_G\left(\fb R[x,f])\right) \quad \mbox{for
all~} \fb \in \mathcal{I}(G).
$$
\end{prop}

The following main result from \cite{ga} will also play a key r\^ole
in this paper.

\begin{thm} [{\cite[Corollary 3.11]{ga}}] \label{ga3.11} Let $G$ be
an $x$-torsion-free left $R[x,f]$-module. Suppose that $G$ is either
Artinian or Noetherian as an $R$-module. Then the set
$\mathcal{I}(G)$ of $G$-special $R$-ideals is finite.
\end{thm}

In \cite[\S 4]{ga}, the author applied Theorem \ref{ga3.11}, in the
case where $(R,\fm)$ is an $F$-injective Gorenstein local ring of
positive dimension $d$, to the `top' local cohomology module $H :=
H^{d}_{\fm}(R)$ of $R$. (It is worth pointing out that a Gorenstein
local ring of characteristic $p$ is $F$-injective if and only if it
is $F$-pure.) The statement that $R$ is $F$-injective implies that
$H$, with its natural structure as a left $R[x,f]$-module, is
$x$-torsion-free, and this implies that $R$ must be reduced. Note
also that, in this case, $H^d_{\fm}(R) \cong E_R(R/\fm)$. Theorem
\ref{ga3.11} yields the finite set $\mathcal{I}(H)$ of radical
ideals of $R$. Let $\fb$ denote the smallest ideal of positive
height in $\mathcal{I}(H)$ (interpret $\height R$ as $\infty$). In
\cite[Corollary 4.7]{ga} it was shown that, if $c$ is any element of
$\fb\cap R^{\circ}$, then $c$ is a test element for $R$, and that
$\fb$ is the test ideal $\tau(R)$ of $R$ (that is (in this case),
the ideal of $R$ generated by all test elements of $R$). These
results were obtained without the assumption that $R$ is excellent.

In \cite{gatcti}, the author generalized the above-described results
of \cite[Corollary 4.7]{ga} to the case where $(R,\fm)$ is local and
$E := E_R(R/\fm)$ carries a structure of $x$-torsion-free left
$R[x,f]$-module. In this situation, Theorem \ref{ga3.11} again
yields the finite set $\mathcal{I}(E)$ of radical ideals of $R$. One
of the main results of \cite{gatcti} is that, if $\fb$ is the
smallest ideal of positive height in $\mathcal{I}(E)$, then each
element of $\fb \cap R^{\circ}$ is a test element for modules for
$R$.

In the case when the local ring $(R,\fm)$ is $F$-pure (and $E$ is as
above), the graded left $R[x,f]$-module $R[x,f]\otimes_RE$ is
$x$-torsion-free; in this paper, we shall use a technique similar to
that of \cite[Theorem 3.5]{gatcti} to show that
$\mathcal{I}(R[x,f]\otimes_RE)$ is a finite set, and this will
enable us to draw conclusions about tight closure test elements.

\begin{ntn}
\label{nt.1} The notation introduced in the Introduction will be
maintained. The symbols $\fa$ and $\fb$ will always denote ideals of
$R$.

For $n \in \Z$, we shall denote the $n$th component of a $\Z$-graded
left $R[x,f]$-module $G$ by $G_n$. If $\phi : L \lra M$ is a
homogeneous homomorphism of $\Z$-graded left $R[x,f]$-modules (of
degree $0$), then the notation $\phi = \bigoplus_{n\in\Z}\phi_n:
\bigoplus_{n\in\Z}L_n \lra \bigoplus_{n\in\Z}M_n$ will indicate that
$\phi_n : L_n \lra M_n$ is the restriction of $\phi$ to $L_n$ (for
all $n \in\Z$). As in \cite[Notation 1.1]{gatcti}, we denote the
{\em $n$th shift functor\/} on the category of $\Z$-graded left
$R[x,f]$-modules and homogeneous homomorphisms (of degree $0$) by $
(\: {\scriptscriptstyle \bullet} \:)(n)$. Note that the identity
functor on this category is the $0$th shift functor.
\end{ntn}

\begin{rmk}
\label{pl.0} Note that $Rx$, which is a left $R$-submodule of
$R[x,f]$, has a structure of right $R$-module via $f$, that is $rxr'
= r'^prx$ for all $r,r' \in R$. Note also that, for a given
$R$-module $M$, there is a bijective correspondence between the set
of structures as left $R[x,f]$-module on $M$ (extending its
$R$-module structure) and the set $$\Hom_R(Rx\otimes_RM, M),$$ under
which such a structure corresponds to the $R$-homomorphism $\phi$
for which $$\phi(rx \otimes m) = rxm \quad \mbox{for all~} r \in R
\mbox{~and~} m \in M.$$ Furthermore, given an $R$-homomorphism
$\theta : Rx\otimes_RM \lra M$, the corresponding left
$R[x,f]$-module structure on $M$ is such that $xm = \theta (x
\otimes m)$ for all $m \in M$.
\end{rmk}

We shall make several uses of the following lemma of M. Hochster and
J. L. Roberts.

\begin{lem}[M. Hochster and J. L. Roberts {\cite[Lemma 6.2 and Corollary 6.13]{HocRob74}}]
\label{mr.HR} Assume that $R$ is $F$-pure. Then, for each $\fp \in
\Spec (R)$,
\begin{enumerate}
\item the localization $R_{\fp}$ is $F$-pure, and
\item the completion $\widehat{R_{\fp}}$ of $R_{\fp}$ is $F$-pure.
\end{enumerate}
\end{lem}

We shall also need the following result due to R. Fedder.

\begin{thm} [R. Fedder {\cite[Theorem 1.12]{Fedde83}}]
\label{mr.1} Suppose that $R = S/\fa$, where $(S,\fn)$ is a regular
local ring of characteristic $p$ and $\fa$ is a proper ideal of $S$.
Then $R$ is $F$-pure if and only if $(\fa^{[p]} : \fa) \not\subseteq
\fn^{[p]}$.
\end{thm}

\section{\sc Homomorphic images of regular local rings of characteristic $p$}
\label{hrlr}

Suppose that $(R,\fm)$ is local. Recall that $E_R(R/\fm)$ has a
natural structure as a module over the completion $\widehat{R}$ of
$R$. Therefore, a structure as left $R[x,f]$-module on $E_R(R/\fm)$
induces, in a unique way, a structure as left
$\widehat{R}[x,f]$-module on it extending its $R[x,f]$-module
structure. As $\widehat{R}$-module, $E_R(R/\fm) \cong
E_{\widehat{R}}(\widehat{R}/\fm \widehat{R})$, and I. S. Cohen's
Structure Theorem for complete local rings ensures that
$\widehat{R}$ is a homomorphic image of a complete regular local
ring of characteristic $p$.

With these considerations in mind, we are going, in this section, to
consider possible left $R[x,f]$-module structures on $E_R(R/\fm)$
when $R$ is a homomorphic image of a regular local ring of
characteristic $p$.

\begin{ntn} \label{hrlr.1} Throughout this section, we shall assume
that $R = S/\fa$, where $S$ is a regular local ring of
characteristic $p$ and $\fa$ is a proper, non-zero ideal of $S$. For
$s \in S$, we shall denote the natural image of $s$ in $R$ by
$\overline{s}$. We shall only assume that $S$ is complete when this
is explicitly stated. We shall use $\fn$ to denote the maximal ideal
of $S$, so that $\fm := \fn/\fa$ is the maximal ideal of $R$. Set $E
:= E_S(S/\fn)$, and note that $(0:_E\fa) \cong E_R(R/\fm)$ as
$R$-modules (by \cite[10.1.15]{LC}, for example). We shall interpret
$(0:_E\fa)$ as $E_R(R/\fm)$.

Observe that $R[x,f]$ is a homomorphic image of $S[x,f]$ under a
homomorphism that extends the natural surjective homomorphism from
$S$ to $R$ and maps the indeterminate $x$ to $x$. Let $y$ be a new
indeterminate. Since $E \cong H^{\dim S}_{\fn}(S)$, the $S$-module
$E$ has a natural structure as a left $S[y,f]$-module (as described
in, for example, \cite[Reminder 4.1]{ga}). For any element $u \in
S$, it is easy to see (for example, by use of \cite[Lemma 1.3]{KS})
that we can endow $E$ with a structure of left $S[x,f]$-module under
which $xe = uye$ for all $e \in E$. A recurring theme of this
section is the idea of trying to choose a $u$ as above in such a way
that $(0:_E\fa) = E_R(R/\fm)$ is an $S[x,f]$-submodule of $E$, and
so becomes a left $R[x,f]$-module.

For an $S$-submodule $M$ of $E$ and $n \in \nn$, we shall use
$Sy^nM$ to denote the $S$-submodule of $E$ generated by $y^nM :=
\left\{y^nm : m \in M\right\}$. Thus
$$
Sy^nM = \left\{\textstyle{\sum_{i=1}^ts_iy^nm_i : t \in \N,~ s_1,
\ldots,s_t \in S,~ m_1, \ldots, m_t \in M} \right\}.
$$

We shall denote $1 + p + p^2 + \cdots +p^{n-1}$, where $n \in \N$,
by $\nu_n$, and we shall interpret $\nu_0$ as $0$. Note that
\begin{enumerate}
\item $p\nu_n < 1 + p\nu_n = \nu_{n+1}$,
\item $(1-p)\nu_n = 1 - p^n$, and
\item $\nu_1 = 1$.
\end{enumerate}
\end{ntn}

\begin{rmk}
\label{hrlr.2} We use the notation of \ref{hrlr.1}. It follows from
G. Lyubeznik and K. E. Smith \cite[Example 3.7]{LyuSmi01} that, when
$S$ is complete, each $S[x,f]$-module structure on $E$ is such that
there exists $u \in S$ for which $xe = uye$ for all $e \in E$.
\end{rmk}

\begin{lem}
\label{hrlr.3} We use the notation of\/ {\rm \ref{hrlr.1}}. Let $n
\in \N$. Then $(0:_SSy^n(0:_E\fa)) = \fa^{[p^n]}$.
\end{lem}

\begin{proof} Since $S$ is regular, the right $S$-module $Sy^n$ is flat, and so the
homomorphism (of left $S$-modules) $Sy^n \otimes _S(0:_E\fa) \lra
Sy^n \otimes_S E$ induced by the inclusion map is injective.
However, by \cite[Remark 4.2(iii)]{ga}, the $S$-homomorphism
$\delta_n:Sy^n \otimes_SE \lra E$ for which $$\delta_n(sy^n \otimes
e) = sy^ne \quad \text{for all~} s \in S \text{~and~} e \in E$$ is
an isomorphism. There is therefore a commutative diagram
\[
\begin{picture}(300,75)(-150,-25)
\put(-55,40){\makebox(0,0){$Sy^n \otimes _S(0:_E\fa)$}}
\put(25,40){\makebox(0,0)[l]{$ Sy^n \otimes_S E$}}
\put(-10,40){\vector(1,0){30}}
\put(-35,10){\makebox(0,0)[l]{$^{\lambda_n }$}}
\put(50,10){\makebox(0,0)[l]{$^{ \delta_n}$}}
\put(34,10){\makebox(0,0)[l]{$^{\cong}$}}
\put(-45,-20){\makebox(0,0){$ Sy^n(0:_E\fa) $}}
\put(10,-14){\makebox(0,0){$^{ \subseteq }$}}
\put(40,-20){\makebox(0,0)[l]{$ E $}}
\put(-15,-20){\vector(1,0){50}} \put(-40,30){\vector(0,-1){40}}
\put(45,30){\vector(0,-1){40}}
\end{picture}
\]
of $S$-modules and $S$-homomorphisms in which the upper horizontal
homomorphism is induced by inclusion and the $S$-homomorphism
$\lambda_n:Sy^n \otimes_S(0:_E\fa) \lra Sy^n(0:_E\fa)$ is such that
$\lambda_n(sy^n \otimes e) = sy^ne$ for all $s \in S$ and $e \in
(0:_E\fa)$. It follows that $\lambda_n$ is an isomorphism, and so
the annihilator of $Sy^n(0:_E\fa)$ is equal to the annihilator of
$Sy^n \otimes _S(0:_E\fa)$. However, since the right $S$-module
$Sy^n$ is flat, there are $S$-isomorphisms
\begin{align*}
Sy^n \otimes _S(0:_E\fa) & \cong Sy^n \otimes _S\Hom_S(S/\fa,E)\\
&\cong \Hom_S(Sy^n \otimes_S S/\fa , Sy^n \otimes_S E)\\ & \cong
\Hom_S(S/\fa^{[p^n]}, E).
\end{align*}
(We have used \cite[Theorem 7.11]{HM} to obtain the second
isomorphism here.) But (even though $S$ might not be complete), an
$S$-module and its Matlis dual have equal annihilators. Therefore
$$(0:_SSy^n(0:_E\fa)) =\left(0:_S\Hom_S(S/\fa^{[p^n]}, E)\right) =
\left(0:_SS/\fa^{[p^n]}\right) = \fa^{[p^n]}.$$
\end{proof}

\begin{lem}
\label{hrlr.4} We use the notation of\/ {\rm \ref{hrlr.1}}. Let $u
\in S$. Put the left $S[x,f]$-module structure on $E$ for which $xe
= uye$ for all $e \in E$. Then $(0:_E\fa)$ is an $S[x,f]$-submodule
of $E$ if and only if $u \in (\fa^{[p]} : \fa)$.

When this condition is satisfied, $E_R(R/\fm)$ is a left
$R[x,f]$-module with $xe = uye$ for all $e \in (0:_E\fa) =
E_R(R/\fm)$ and $re = se$ for any $r \in R$ and $s \in S$ for which
$\overline{s} = r$.
\end{lem}

\begin{note} In the special case of Lemma \ref{hrlr.4} in which $S$
is complete, the results of the lemma can be approached by means of
the techniques developed by M. Katzman in \cite[\S\S3,4]{MK}.
However, we now have applications in mind where the full generality
of Lemma \ref{hrlr.4} will be required.
\end{note}

\begin{proof} $(\Leftarrow)$ Let $e \in (0:_E\fa)$. We must show that $xe = uye
\in (0:_E\fa)$. So let $s \in \fa$. Since $u \in (\fa^{[p]} : \fa)$,
there exist $a_1, \ldots, a_t \in \fa$ and $s_1, \ldots, s_t \in S$
such that $su = {\textstyle \sum_{i=1}^t s_ia_i^p}. $ Then
$$
sxe = suye = \left({\textstyle \sum_{i=1}^t s_ia_i^p}\right)ye =
{\textstyle \sum_{i=1}^t s_iya_ie} = 0,
$$
since $a_ie = 0$ for all $i = 1, \ldots, t$. Thus $xe \in
(0:_E\fa)$. It follows that $(0:_E\fa)$ is an $S[x,f]$-submodule of
$E$.

$(\Rightarrow)$ Assume that $(0:_E\fa)$ is an $S[x,f]$-submodule of
$E$, so that $suye = 0$ for all $s \in \fa$ and $e \in (0:_E\fa)$.
Thus, on use of Lemma \ref{hrlr.3}, we see that $$u\fa \subseteq
\left(0:_SSy(0:_E\fa)\right) = \fa^{[p]}.$$ Therefore $u \in
(\fa^{[p]} : \fa)$.

The claims in the final paragraph are easy to check.
\end{proof}

\begin{lem}[M. Blickle {\cite[Proposition 3.36]{Blickle}}]
\label{hrlr.4a} We use the notation of\/ {\rm \ref{hrlr.1}}. Suppose
that $S$ is complete, and that $(0:_E\fa)$ has a left
$R[x,f]$-module structure that extends its $R$-module structure.
Then there exists $u \in (\fa^{[p]} : \fa)$ such that $xe = uye$ for
all $e \in (0:_E\fa)$.
\end{lem}

\begin{note} This result is included in M. Blickle's PhD
dissertation, but we give a short, self-contained proof for the
convenience of the reader.
\end{note}

\begin{proof} There is a natural surjective ring homomorphism
$S[x,f] \lra R[x,f]$ which extends the natural ring homomorphism $S
\lra R$ and maps $x$ to $x$. We can use this to consider $(0:_E\fa)$
as an $S[x,f]$-module. By Remark \ref{pl.0}, there is therefore an
$S$-homomorphism $\delta : Sx \otimes_S(0:_E\fa) \lra (0:_E\fa)$
such that $\delta(sx \otimes e) = sxe$ for all $s \in S$ and $e \in
(0:_E\fa)$.

Since $S$ is regular, the right $S$-module $Sx$ is flat; therefore
the $S$-homomorphism $j : Sx \otimes_S(0:_E\fa) \lra Sx \otimes_SE$
induced by inclusion is injective. As $E$ is injective as
$S$-module, there is an $S$-homomorphism $\delta': Sx \otimes_SE
\lra E$ which makes the diagram
\[
\begin{picture}(300,75)(-150,-25)
\put(-55,40){\makebox(0,0){$Sx \otimes _S(0:_E\fa)$}}
\put(0,46){\makebox(0,0){$^{\delta}$}}
\put(25,40){\makebox(0,0)[l]{$ (0:_E\fa)$}}
\put(-10,40){\vector(1,0){30}} \put(-35,10){\makebox(0,0)[l]{$^{j
}$}} \put(50,10){\makebox(0,0)[l]{$^{ \subseteq}$}}
\put(-45,-20){\makebox(0,0){$ Sx \otimes _SE $}}
\put(0,-14){\makebox(0,0){$^{ \delta' }$}}
\put(40,-20){\makebox(0,0)[l]{$ E $}}
\put(-15,-20){\vector(1,0){50}} \put(-40,30){\vector(0,-1){40}}
\put(45,30){\vector(0,-1){40}}
\end{picture}
\]
commute.

In view of Remark \ref{pl.0}, there is therefore a structure of left
$S[x,f]$-module on $E$ which is such that $xe = \delta'(x \otimes
e)$ for all $e \in E$; this structure extends the $S[x,f]$-module
structure on $(0:_E\fa)$.

By Lyubeznik's and Smith's result quoted in Remark \ref{hrlr.2},
there exists $u \in S$ for which $xe = uye$ for all $e \in E$, and,
in particular, for all $e \in (0:_E\fa)$. We can now appeal to Lemma
\ref{hrlr.4} to see that $u \in (\fa^{[p]}:\fa)$.
\end{proof}

\begin{rmk} \label{hrlr.5} We again use the notation of\/ {\rm
\ref{hrlr.1}}. In the special case in which $S$ is complete, it
follows from Matlis duality (see, for example, \cite[p.\ 154]{SV})
that each $S$-submodule $M$ of $E$ satisfies $M = (0:_E(0:_SM))$,
and so has the form $(0:_E\fk)$ for a (uniquely determined) ideal
$\fk$ of $S$.
\end{rmk}

\begin{prop}
\label{hrlr.6} We use the notation of\/ {\rm \ref{hrlr.1}}. Let $u
\in (\fa^{[p]} : \fa)$, and put the $S[x,f]$-module structure on $E$
for which $xe = uye$ for all $e \in E$, so that $(0:_E\fa)$ is an
$S[x,f]$-submodule of $E$, by Lemma\/ {\rm \ref{hrlr.4}}. Then
$$
\grann_{S[x,f]}(0:_E\fa) = \bigoplus_{n \in \nn}
\left(\fa^{[p^n]}:u^{\nu_n}\right)x^n.
$$
\end{prop}

\begin{proof} We can write $\grann_{S[x,f]}(0:_E\fa) = \bigoplus_{n \in
\nn} \fb_nx^n$ for the appropriate ascending chain
$(\fb_n)_{n\in\nn}$ of ideals of $S$. In fact, $$\fb_n = \left\{s
\in S : sx^n \in \ann_{S[x,f]}(0:_E\fa)\right\} \quad \text{for
all~} n \in \nn.$$ But, for $n \in \N$, $s \in S$ and $e \in E$,
$$
sx^ne = s(uy)^ne = su^{1+p+\cdots+p^{n-1}}y^ne = su^{\nu_n}y^ne.
$$
Therefore $\fb_0 = (0:_S(0:_E\fa)) = \fa =
\left(\fa^{[p^0]}:u^{\nu_0}\right)$ and, by Lemma \ref{hrlr.3},
$$\fb_n = \left\{s \in S: su^{\nu_n} \in \left(0:_SSy^n(0:_E\fa)\right) =
\fa^{[p^n]}\right\} = \left(\fa^{[p^n]}:u^{\nu_n}\right) \quad
\text{for all~} n \in \N.$$
\end{proof}

Next, we record some facts concerning primary decompositions of the
ideal $(\fa^{[p]} : \fa)$.

\begin{prop}
\label{hrlr.7} We use the notation of\/ {\rm \ref{hrlr.1}}. Let
$\fb_1, \ldots, \fb_t,\fc$ be ideals of $S$, and let $\fa = \fq_1
\cap \ldots \cap \fq_t$ be a minimal primary decomposition of $\fa$.
\begin{enumerate}
\item We have $(\fb_1 \cap \cdots \cap \fb_t)^{[p^n]} = \fb_1 ^{[p^n]}\cap \cdots \cap
\fb_t^{[p^n]}$ for all $n \in \N$.
\item If $\fq$ is a $\fp$-primary ideal of $S$, then $\fq^{[p]}$ is also
$\fp$-primary.
\item If
$n \in \nn$, then $\fa^{[p^n]} = \fq_1^{[p^n]} \cap \cdots \cap
\fq_t^{[p^n]}$ is a minimal primary decomposition of $\fa^{[p^n]}$.
\item We have $(\fa :\fc)^{[p]} = (\fa ^{[p]}:\fc^{[p]})$ and
$(\fa^{[p]} : \fa) \subseteq \left((\fa :\fc)^{[p]} :
(\fa : \fc)\right)$.
\item If $\fp$ is an associated prime ideal of $\fa$, then $(\fa^{[p]} : \fa)
\subseteq (\fp^{[p]} : \fp)$.
\item Since $0 \neq \fa \neq S$, we have $(\fa^{[p]} : \fa) \neq S$.
If $\fp_1 := \sqrt{\fq_1}$ is a minimal prime ideal of $\fa$, then
$\fp_1$ is a minimal prime ideal of $(\fa^{[p]} : \fa)$ and the
unique $\fp_1$-primary component of $(\fa^{[p]} : \fa)$ is
$(\fq_1^{[p]} : \fq_1)$.
\item If $\fq$ is a $\fp$-primary irreducible ideal of $S$, then $\fq^{[p]}$ is also $\fp$-primary
and irreducible.
\end{enumerate}
\end{prop}

\begin{proof} (i), (ii), (iii) These are immediate consequences of
the fact that the Frobenius homomorphism is a faithfully flat ring
homomorphism from $S$ to $S$.

(iv) That $(\fa :\fc)^{[p]} = (\fa ^{[p]}:\fc^{[p]})$ also follows
from the faithful flatness of the Frobenius ring homomorphism on
$S$. Furthermore, since $(\fa^{[p]} : \fa)(\fa : \fc)\fc \subseteq
\fa^{[p]}$, it is clear that $(\fa^{[p]} : \fa)(\fa : \fc)\fc^{[p]}
\subseteq \fa^{[p]}$, and therefore $$(\fa^{[p]} : \fa)(\fa : \fc)
\subseteq (\fa ^{[p]}:\fc^{[p]}) = (\fa :\fc)^{[p]}.$$

(v) This is immediate from part (iv), since there exists $s \in S$
such that $\fp = (\fa : s)$.

(vi) If it were the case that $\fa^{[p]} = \fa$, then we should have
$\fn\fa = \fa$, so that $\fa$ would be zero, by Nakayama's Lemma;
therefore $(\fa^{[p]} : \fa) \neq S$.

It follows from part (iii) that
$$
(\fa^{[p]} : \fa) =(\fq_1^{[p]} \cap \cdots \cap \fq_t^{[p]}:\fa) =
(\fq_1^{[p]}:\fa) \cap \cdots \cap (\fq_t^{[p]}:\fa),
$$
and from part (ii) that this becomes a primary decomposition of
$(\fa^{[p]} : \fa)$ after discard of terms equal to $S$, since, for
each $i=1, \ldots,t$, the ideal $(\fq_i^{[p]}:\fa)$ is either
$\sqrt{\fq_i}$-primary or equal to $S$.

For the remainder of this proof, for convenience, we denote $\fp_1$
by $\fp$ and $\fq_1$ by $\fq$. Let $\phantom{\fa}^e$ and
$\phantom{\fa}^c$ stand for extension and contraction of ideals with
respect to the natural ring homomorphism $S \lra S_{\fp}$. Then
$(\fq_i^{[p]}:\fa)^{ec} = S$ for all $i = 2, \ldots, t$, since $\fp$
is a minimal prime ideal of $\fa$; therefore
$$
\left(\fa^{[p]}:\fa\right)^{ec}= \left(\fq^{[p]}:\fa\right)^{ec} =
\left((\fq^{[p]})^e:\fa^e\right)^{c} =
\left((\fq^{[p]})^e:\fq^e\right)^{c} =
\left(\fq^{[p]}:\fq\right)^{ec} = \left(\fq^{[p]}:\fq\right),
$$
since $\left(\fq^{[p]}:\fq\right)$ is $\fp$-primary by part (ii) and
the first paragraph of the proof of this part (vi).

We can now conclude from the last two paragraphs that $\fp$ is a
minimal prime ideal of $(\fa^{[p]} : \fa)$ and the unique
$\fp$-primary component of $(\fa^{[p]} : \fa)$ is $(\fq^{[p]} :
\fq)$.

(vii) This was established in \cite[Remark 8.3(iii)]{SN} by use of a
result of Huneke and Sharp \cite[Proposition 1.5]{58}.
\end{proof}

\begin{prop}
\label{hrlr.8} We use the notation of\/ {\rm \ref{hrlr.1}}. Let $\fp
\in \Spec (S)$ and $u \in (\fp^{[p]} : \fp) \setminus \fp^{[p]}$.
Put the $S[x,f]$-module structure on $E$ for which $xe = uye$ for
all $e \in E$, so that $(0:_E\fp)$ is an $S[x,f]$-submodule of $E$,
by Lemma\/ {\rm \ref{hrlr.4}}.

Then $(\fp^{[p^n]} : u^{\nu_n}) = \fp$ for all $n \in \nn$, so that
$(0:_E\fp)$ is not $x$-torsion.
\end{prop}

\begin{proof} The claim that $(\fp^{[p^n]} : u^{\nu_n}) = \fp$ is obvious
when $\fp = 0$ or $n = 0$;
we therefore suppose that $\fp \neq 0$ and $n > 0$.

By Proposition \ref{hrlr.7}(ii), the Frobenius power $\fp^{[p]}$ is
$\fp$-primary; since $u \not\in \fp^{[p]}$, the ideal
$(\fp^{[p]}:u)$ is $\fp$-primary, so that $(\fp^{[p]}:u) \subseteq
\fp$. But $u\fp \subseteq \fp^{[p]}$, and so $$\fp \subseteq
(\fp^{[p]}:u) \subseteq \fp.$$ Therefore $(\fp^{[p]}:u) = \fp$ and
the desired result has been proved in the case where $n = 1$.

Now suppose, inductively, that $n \geq 1$ and we have shown that
$(\fp^{[p^n]} : u^{\nu_n}) = \fp$. We can use Proposition
\ref{hrlr.7}(iv) to deduce that
$$
\fp^{[p]} = (\fp^{[p^n]} : u^{\nu_n})^{[p]} = ((\fp^{[p^n]})^{[p]} :
(u^{\nu_n})^p) = (\fp^{[p^{n+1}]} : u^{\nu_np}).
$$
But $\nu_np + 1 = \nu_{n+1}$, and so
$$
(\fp^{[p^{n+1}]} : u^{\nu_{n+1}}) = (\fp^{[p^{n+1}]} :
u^{\nu_{n}p}u) = ((\fp^{[p^{n+1}]} : u^{\nu_np}):u) = (\fp^{[p]}: u)
= \fp
$$
by the case where $n = 1$. This completes the inductive step.

It now follows from Proposition \ref{hrlr.6} that
$\grann_{S[x,f]}(0:_E\fp) = \bigoplus_{n\in\nn}\fp x^n$. Note that
$(0:_E\fp)$ is Artinian as $S$-module. We now see from the
Hartshorne--Speiser--Lyubeznik Theorem (see G. Lyubeznik
\cite[Proposition 4.4]{Lyube97} and R. Hartshorne and R. Speiser
\cite[Proposition 1.11]{HarSpe77}) that if $(0:_E\fp)$ were
$x$-torsion, then it would be annihilated by $x^h$ for some $h \in
\N$. As this is not the case, $(0:_E\fp)$ is not $x$-torsion.
\end{proof}

\begin{prop}
\label{hrlr.9} Again use the notation of\/ {\rm \ref{hrlr.1}}.
Assume that $S$ is complete and the (complete) local ring $R =
S/\fa$ is Gorenstein. Then $(0:_E\fa) = E_R(R/\fm) \cong H^{\dim
R}_{\fm}(R)$, and so has a natural structure as a left
$R[x,f]$-module. By Lemma\/ {\rm \ref{hrlr.4a}}, there exists $u \in
(\fa^{[p]} : \fa)$ such that, under this natural structure, $xe =
uye$ for all $e \in (0:_E\fa)$. Then $(\fa^{[p]}:\fa) = \fa^{[p]} +
Su$.
\end{prop}

\begin{proof} Let $z$ be a new variable. It follows from Lyubeznik--Smith
\cite[Example 3.7]{LyuSmi01} that, since $R$ is complete and
Gorenstein, each $R[z,f]$-module structure on $(0:_E\fa)$ is such
that there exists $w' \in S$ for which $ze = \overline{w'}xe$ for
all $e \in (0:_E\fa)$.

Let $v \in (\fa^{[p]}:\fa)$. By Lemma \ref{hrlr.4}, there is a left
$R[z,f]$-module structure on $(0:_E\fa)$ for which $ze = vye$ for
all $e \in (0:_E\fa)$. By the first paragraph of this proof, there
exists $w \in S$ for which $ze = \overline{w}xe$ for all $e \in
(0:_E\fa)$. Therefore $vye = ze = wxe = wuye$ for all $e \in
(0:_E\fa)$, and $v - wu$ annihilates $Sy(0:_E\fa)$, so that $v - wu
\in \fa^{[p]}$ by Lemma \ref{hrlr.3}. Therefore $(\fa^{[p]}:\fa)
\subseteq \fa^{[p]} + Su$; the reverse inclusion is obvious.
\end{proof}

\begin{cor}
\label{hrlr.10} Once again, we use the notation of\/ {\rm
\ref{hrlr.1}}. The minimum number of generators of the $S$-module
$(\fa^{[p]}:\fa)/\fa^{[p]}$ is equal to the minimum number of
generators of the $\widehat{S}$-module
$\left((\fa\widehat{S})^{[p]}:\fa\widehat{S}\right)/(\fa\widehat{S})^{[p]}$.

Consequently, if the local ring $R = S/\fa$ is Gorenstein, then
$(\fa^{[p]}:\fa)/\fa^{[p]}$ is a cyclic $S$-module.
\end{cor}

\begin{proof} Let $t$ denote the minimal number of generators of the $S$-module
$(\fa^{[p]}:\fa)/\fa^{[p]}$; then $t$ is equal to the dimension of
$(\fa^{[p]}:\fa)/\left(\fn(\fa^{[p]}:\fa) + \fa^{[p]}\right)$ as a
vector space over $S/\fn$. As $\widehat{S}$ is faithfully flat over
$S$, it is straightforward to check that this dimension is equal to
the dimension, as a vector space over $\widehat{S}/\fn\widehat{S}$,
of
$$\left((\fa\widehat{S})^{[p]}:\fa\widehat{S}\right)/\left(\fn\widehat{S}((\fa\widehat{S})^{[p]}:\fa\widehat{S})
+ (\fa\widehat{S})^{[p]}\right).$$ Therefore the minimum number of
generators of the $\widehat{S}$-module
$\left((\fa\widehat{S})^{[p]}:\fa\widehat{S}\right)/(\fa\widehat{S})^{[p]}$
is also equal to $t$.

Note that $\widehat{R} = \widehat{S}/\fa \widehat{S}$ is Gorenstein
if and only if $R$ is Gorenstein. Therefore, in order to prove the
claim in the final paragraph of the corollary, it is sufficient for
us to do so in the case where $S$ is complete, and, in that case,
the claim is immediate from Proposition \ref{hrlr.9}.
\end{proof}

\begin{prop}
\label{hrlr.11} Use the notation of\/ {\rm \ref{hrlr.1}}. Suppose
that $u_1, u_2 \in S$ are such that
$$
(\fa^{[p]}:\fa) = \fa^{[p]} + Su_1 =  \fa^{[p]} + Su_2.
$$
Bearing in mind Lemma\/ {\rm \ref{hrlr.4}}, let $D_i$, for $i =
1,2$, denote $(0:_E\fa)$ endowed with the left $R[x,f]$-module
structure that (extends its $R$-module structure and) is such that
$xe = u_iye$ for all $e \in (0:_E\fa)$. Then a subset of\/
$(0:_E\fa)$ is an $R[x,f]$-submodule of $D_1$ if and only if it is
an $R[x,f]$-submodule of $D_2$.
\end{prop}

\begin{proof} There exist $t \in \fa^{[p]}$ and $s \in S$ such that
$u_1 = t + su_2$, and $a_1, \ldots, a_h \in \fa$ and $s_1, \ldots,
s_h \in S$ such that $ t = {\textstyle \sum_{i=1}^h s_ia_i^p}$;
therefore, for $e \in (0:_E\fa)$, we have
$$
tye = \left({\textstyle \sum_{i=1}^h s_ia_i^p}\right)ye =
{\textstyle \sum_{i=1}^h s_ia_i^pye} = {\textstyle \sum_{i=1}^h
s_iya_ie} = 0.
$$
Hence $u_1ye = su_2ye$ for all $e \in (0:_E\fa)$, so that every
$R[x,f]$-submodule of $D_2$ is an $R[x,f]$-submodule of $D_1$. We
can similarly prove that every $R[x,f]$-submodule of $D_1$ is an
$R[x,f]$-submodule of $D_2$.
\end{proof}

\begin{cor}
\label{hrlr.12} Again use the notation of\/ {\rm \ref{hrlr.1}}, so
that $(R,\fm)$ is local and equal to the homomorphic image $S/\fa$
of the regular local ring $S$.

Suppose that $u \in S$ is such that $ (\fa^{[p]}:\fa) = \fa^{[p]} +
Su.$ Bearing in mind Lemma\/ {\rm \ref{hrlr.4}}, let $D$ denote
$(0:_E\fa)$ endowed with the left $R[x,f]$-module structure that
(extends its $R$-module structure and) is such that $xe = uye$ for
all $e \in (0:_E\fa)$. If $R$ is regular, then $D$ is a simple
$R[x,f]$-module.
\end{cor}

\begin{proof} Let $d$ denote $\dim R$. Since $E_R(R/\fm)$ has a
natural structure as a module over the completion $\widehat{R}$ of
$R$, the specified structure as left $R[x,f]$-module on $E_R(R/\fm)$
induces, in a unique way, a structure as left
$\widehat{R}[x,f]$-module on it extending its $R[x,f]$-module
structure. Note that $\widehat{R} = \widehat{S}/\fa \widehat{S}$,
that $\left((\fa \widehat{S})^{[p]}:_{\widehat{S}}\fa
\widehat{S}\right) = (\fa \widehat{S})^{[p]} + \widehat{S}u$, and
that $$(0:_E\fa) = (0:_E\fa\widehat{S}) =
E_{\widehat{R}}(\widehat{R}/\fm\widehat{R}).$$

Because $\widehat{R}$ is regular, $(0:_E\fa\widehat{S}) \cong
H^{d}_{\fm\widehat{R}}(\widehat{R})$ as $\widehat{R}$-modules. The
natural left $\widehat{R}[x,f]$-module structure on this can, by
Lemma \ref{hrlr.4a}, be described by some $\widehat{v} \in
\left((\fa \widehat{S})^{[p]}:_{\widehat{S}}\fa \widehat{S}\right)$,
in the sense that, under this natural structure, $xe =
\widehat{v}ye$ for all $e \in (0:_E\fa\widehat{R})$. By Proposition
\ref{hrlr.9}, $\left((\fa \widehat{S})^{[p]}:_{\widehat{S}}\fa
\widehat{S}\right) = (\fa \widehat{S})^{[p]} +
\widehat{S}\widehat{v}$.

Since a subset of $(0:_E\fa\widehat{S})$ is an $R$-submodule if and
only if it is an $\widehat{R}$-submodule, it follows from
Proposition \ref{hrlr.11} that $D$ is simple as a left
$R[x,f]$-module if and only if $H^{d}_{\fm\widehat{R}}(\widehat{R})$
is simple as a left $\widehat{R}[x,f]$-module under its natural
structure. But, as $\widehat{R}$ is excellent (as it is complete)
and regular, this is immediate from a result of K. E. Smith
\cite[Theorem 2.6]{Smi97}.
\end{proof}

\section{\sc Applications to $F$-pure local rings}\label{fp}

In \cite[Theorem 3.5]{gatcti}, the author showed that, if $(R,\fm)$
is local and the $R$-module structure on $E_R(R/\fm)$ can be
extended to an $x$-torsion-free left $R[x,f]$-module structure, then
$R$ is $F$-pure. The first two results in this section establish the
converse.

\begin{prop}
\label{fp.0} Use the notation of\/ {\rm \ref{hrlr.1}}. Suppose that
$R$ is $F$-pure, so that, by Fedder's Theorem\/ {\rm \ref{mr.1}},
there exists $u \in (\fa^{[p]}:\fa) \setminus \fn^{[p]}$. Put the
left $S[x,f]$-module structure on $E$ for which $xe = uye$ for all
$e \in E$, so that $(0:_E\fa)= E_R(R/\fm)$ is, by Lemma\/ {\rm
\ref{hrlr.4}}, an $S[x,f]$-submodule of $E$ which inherits a left
$R[x,f]$-module structure (that extends its $R$-module structure).

With this left $R[x,f]$-module structure, $E_R(R/\fm)$ is
$x$-torsion-free.
\end{prop}

\begin{proof} As in the first paragraph of the proof of Corollary
\ref{hrlr.12}, we can extend the specified $R[x,f]$-module structure
on $E_R(R/\fm)$ to an $\widehat{R}[x,f]$-module structure, in a
unique way: note that $\widehat{R}$ is $F$-pure, by Lemma
\ref{mr.HR}, and $u \in \left((\fa
\widehat{S})^{[p]}:_{\widehat{S}}\fa \widehat{S}\right) \setminus
(\fn \widehat{S})^{[p]}$. Thus we can, and do, assume that $R$ and
$S$ are complete.

It now follows from Remark \ref{hrlr.5} that, if $\fb$ denotes
$(0:_S\Gamma_x(E_R(R/\fm)))$, then $\Gamma_x(E_R(R/\fm)) = (0
:_E\fb)$. We suppose that $\fb$ is a proper ideal of $S$, and seek a
contradiction. Since $\Gamma_x(E_R(R/\fm))$ is an $S[x,f]$-submodule
of $E$, we conclude from Lemma \ref{hrlr.4} that $u \in
(\fb^{[p]}:\fb)$. Let $\fp$ be a minimal prime ideal of $\fb$. Then
it follows from Proposition \ref{hrlr.7}(v) that $u \in
(\fp^{[p]}:\fp)$, so that $(0 :_E\fp)$ is an $S[x,f]$-submodule of
$E$, by Lemma \ref{hrlr.4}. Since $\fp \supseteq \fb$, we have $(0
:_E\fp) \subseteq (0 :_E\fb) = \Gamma_x(E_R(R/\fm))$; therefore $(0
:_E\fp)$ is an $x$-torsion $S[x,f]$-submodule of $E$.

We can now use Proposition \ref{hrlr.8} to deduce that $u \in
\fp^{[p]}$, so that, since $\fp \subseteq \fn$, we must have $u \in
\fn^{[p]}$. This contradicts the choice of $u$. Hence $\fb = S$ and
$E_R(R/\fm)$ is an $x$-torsion-free left $R[x,f]$-module.
\end{proof}

\begin{thm}
\label{fp.1} Suppose that $(R,\fm)$ is local. Then $R$ is $F$-pure
if and only if the $R$-module structure on $E_R(R/\fm)$ can be
extended to an $x$-torsion-free left $R[x,f]$-module structure.
\end{thm}

\begin{proof} ($\Leftarrow$) This was proved in \cite[Theorem
3.5]{gatcti}.

($\Rightarrow$) Assume that $R$ is $F$-pure; note that, by Lemma
\ref{mr.HR}, this implies that the completion $\widehat{R}$ of $R$
is $F$-pure. Since $E_R(R/\fm)$ has a natural structure as a module
over $\widehat{R}$, and as there is an $\widehat{R}$-isomorphism
$E_R(R/\fm) \cong E_{\widehat{R}}(\widehat{R}/\fm \widehat{R})$, we
can, and do, assume that $R$ is complete.

We now appeal to Cohen's Structure Theorem for complete local rings
containing a subfield and write $R = S/\fa$, where $S$ is a regular
local ring of characteristic $p$ and $\fa$ is a proper, non-zero
ideal of $S$. It now follows from Proposition \ref{fp.0} that the
$R$-module structure on $E_R(R/\fm)$ can be extended to an
$x$-torsion-free left $R[x,f]$-module structure.
\end{proof}

\begin{cor}
\label{fp.2} Suppose that $(R,\fm)$ is local and $F$-pure. Then $R$
(is reduced and) has a tight closure test element for modules.
\end{cor}

\begin{note} It should be noted that, in Corollary \ref{fp.2}, we have not assumed that $R$ is
excellent.
\end{note}

\begin{proof} By Theorem \ref{fp.1}, the $R$-module structure on $E_R(R/\fm)$ can be
extended to an $x$-torsion-free left $R[x,f]$-module structure. It
now follows from \cite[Theorem 3.5(i)]{gatcti} that $R$ (is reduced
and) has a tight closure test element for modules.
\end{proof}

One of the main aims of this paper is to prove that if $R$ is (not
necessarily local but) excellent and $F$-pure, then $R$ has a big
tight closure test element. In fact, we shall show in \S \ref{aet}
that if $c \in R^{\circ}$ is such that $R_c$ is regular, then $c$
itself is a big test element for $R$. In this section, we shall
concentrate on the case where $(R,\fm)$ is local (and $F$-pure and
excellent).

\begin{rmks}
\label{fp.21} Suppose that $R$ is excellent and reduced.
\begin{enumerate}
\item There exists $c \in R^{\circ}$ such that $R_c$ is regular.
This can be proved easily by prime avoidance arguments, based on the
facts that  $$\Reg (R) := \{ \fp \in \Spec(R) : R_{\fp} \mbox{~is a
regular local ring}\}$$ is open in $\Spec (R)$, and, for each
minimal prime ideal $\fq$ of $R$, the localization $R_{\fq}$ is a
field and so regular.
\item Since $R$ is reduced and excellent, the
concepts of test element for modules and test element for ideals for
$R$ coincide: see \cite[Discussion (8.6) and Proposition
(8.15)]{HocHun90}. If $(R,\fm)$ is local and $c \in R^{\circ}$ is a
test element for the completion $\widehat{R}$ of $R$, then $c$ is a
test element for $R$.
\end{enumerate}
\end{rmks}

\begin{lem}
\label{fp.23} Use the notation of\/ {\rm \ref{hrlr.1}}, and suppose
that $R$ is $F$-pure. Let $t$ denote the minimum number of
generators of the $S$-module $T := (\fa^{[p]}:\fa)/\fa^{[p]}$. Then
it is possible to find $t$ elements $u_1, \ldots, u_t \in
(\fa^{[p]}:\fa) \setminus \fn^{[p]}$ whose natural images in $T$
generate $T$.
\end{lem}

\begin{proof} Suppose, inductively, that $j \in \nn$ is such that $j
< t$ and we have constructed $u_1, \ldots, u_j \in (\fa^{[p]}:\fa)
\setminus \fn^{[p]}$ whose natural images in $T/\fn T$ are linearly
independent over the field $S/\fn$. This is certainly true when $j =
0$. Note that $(\fa^{[p]}:\fa) \not\subseteq \fn^{[p]}$, by Fedder's
Theorem \ref{mr.1}. By ideal avoidance, there exists
$$
u_{j+1} \in (\fa^{[p]}:\fa) \setminus \left(\fn^{[p]} \cup \left(
\fn(\fa^{[p]}:\fa) + \fa^{[p]} + Su_1 + \cdots + Su_j\right)\right),
$$
(or else the vector space $T/\fn T$ over $S/\fn$ would have
dimension less than $t$). Then the natural images of $u_1, \ldots,
u_j, u_{j+1}$ in $T/\fn T$ are linearly independent over $S/\fn$.
This completes the inductive step.

Hence, by induction, we can construct $u_1, \ldots, u_t \in
(\fa^{[p]}:\fa) \setminus \fn^{[p]}$ whose natural images in $T/\fn
T$ are linearly independent over the field $S/\fn$; the natural
images of $u_1, \ldots, u_t$ in $T$ generate this $S$-module.
\end{proof}

\begin{prop}
\label{fp.22} Use the notation of\/ {\rm \ref{hrlr.1}}, and suppose
that $R$ is $F$-pure and that $S$ (and $R$) are complete. Use
Lemma\/ {\rm \ref{fp.23}} to find $u_1, \ldots, u_t \in
(\fa^{[p]}:\fa) \setminus \fn^{[p]}$ such that their natural images
in $(\fa^{[p]}:\fa)/\fa^{[p]}$ form a generating set for this
$S$-module. Bearing in mind Lemma\/ {\rm \ref{hrlr.4}}, for each $i
= 1, \ldots, t$, let $E_i$ denote $(0:_E\fa) = E_R(R/\fm)$ endowed
with the left $R[x,f]$-module structure for which $xe = u_iye$ for
all $e \in (0:_E\fa)$, and note that $E_i$ is $x$-torsion-free, by
Proposition\/ {\rm \ref{fp.0}}; let $\fb_i$ be the unique smallest
ideal of positive height in ${\mathcal I}(E_i)$.

Set $\fb = \fb_1 + \cdots + \fb_t$, and let $c \in R$ be such that
$R_c$ is regular. Then $\fb R_c = R_c$.
\end{prop}

\begin{proof} Let $\fp \in \Spec (R)$ be such that $c \not\in \fp$.
Then $R_{\fp}$, being a localization of $R_c$, is regular. Let $\fq$
be the prime ideal of $S$ that contains $\fa$ and is such that $\fp
= \fq/\fa$. Note that
$$
\left((\fa S_{\fq})^{[p]} : \fa S_{\fq}\right) =  (\fa
S_{\fq})^{[p]} + (u_1/1)S_{\fq} + \cdots + (u_t/1)S_{\fq}.
$$
Since $S_{\fq}/\fa S_{\fq} \cong R_{\fp}$ is a regular local ring,
it follows from Corollary \ref{hrlr.10} that the $S_{\fq}$-module
$\left((\fa S_{\fq})^{[p]} : \fa S_{\fq}\right)/\fa S_{\fq}^{[p]}$
is cyclic; therefore, there exists $i \in \{1,\ldots,t\}$ such that
$$
\left((\fa S_{\fq})^{[p]} : \fa S_{\fq}\right) =  \fa S_{\fq}^{[p]} + (u_i/1)S_{\fq}.
$$
Let $G$ denote the injective envelope, over $R_{\fp}$, of the simple
$R_{\fp}$-module, and use Lemma \ref{hrlr.4} in conjunction with the
isomorphism $R_{\fp} \cong S_{\fq}/\fa S_{\fq}$ and the element
$u_i/1$ to put a left $R_{\fp}[x,f]$-module structure on $G$.

Let $\fk_i$ be the ideal of $S$ that contains $\fa$ and is such that
$\fb_i = \fk_i/\fa$. Since $R$ is complete, it follows from
\cite[Proposition 1.8]{gatcti} that
$$
(0:_E\fk_i) = (0:_{E_R(R/\fm)}\fb_i) = \ann_{E_i}(\fb_iR[x,f]),
$$
so that this is an $R[x,f]$-submodule of $E_i$. It therefore follows
from Lemma \ref{hrlr.4} that $u_i \in
\left((\fk_i)^{[p]}:\fk_i\right)$. Therefore, in $S_{\fq}$, we have
$$
u_i/1 \in \left((\fk_iS_{\fq})^{[p]}:\fk_iS_{\fq}\right),
$$
and so $(0:_G \fk_iS_{\fq}/\fa S_{\fq}) = (0:_G\fb_iR_{\fp})$ is an
$R_{\fp}[x,f]$-submodule of $G$, again by Lemma \ref{hrlr.4}. Now
$G$ is simple as $R_{\fp}[x,f]$-module, by Corollary \ref{hrlr.12}.
Therefore, since a module over a (not necessarily complete) local
ring has the same annihilator as its Matlis dual, $\fb_iR_{\fp}$
must be $0$ or $R_{\fp}$. Since $\height \fb_iR_{\fp} \geq 1$, we
must have $\fb_iR_{\fp} = R_{\fp}$.

Therefore $\fb R_{\fp} = R_{\fp}$. As the later equation is true for
all $\fp \in \Spec (R)$ such that $c \not\in \fp$, we see that $\fb
R_c = R_c$.
\end{proof}

\begin{thm}
\label{fp.24} Suppose that $(R,\fm)$ is local, $F$-pure and
excellent. Let $c \in R^{\circ}$ be such that $R_c$ is regular. (By
Remark\/ {\rm \ref{fp.21}(i)}, it is possible to find such a $c$.)
Then $c$ is a test element for both $R$ and $\widehat{R}$.
\end{thm}

\begin{proof} By Remark \ref{fp.21}(ii), it is sufficient to prove
that $c$ is a test element for $\widehat{R}$. Note that
$\widehat{R}$ is $F$-pure, by Lemma \ref{mr.HR}. Complete local
rings are always excellent. Furthermore, the fibre rings of the flat
ring homomorphism $R_c \lra \widehat{R}_c$ induced by inclusion are
rings of fractions of the formal fibres of $R$, and so are regular;
it therefore follows that $\widehat{R}_c$ is regular. Thus we can,
and do, assume for the remainder of this proof that $R$ is complete.

We now appeal to Cohen's Structure Theorem for complete local rings
containing a subfield and write $R = S/\fa$, where $S$ is a complete
regular local ring of characteristic $p$ and $\fa$ is a proper,
non-zero ideal of $S$. This is consistent with the notation of
\ref{hrlr.1}, and we shall use that notation for the remainder of
this proof.

We now appeal to Proposition \ref{fp.22}. Use Lemma\/ {\rm
\ref{fp.23}} to find $u_1, \ldots, u_t \in (\fa^{[p]}:\fa) \setminus
\fn^{[p]}$ such that their natural images in
$(\fa^{[p]}:\fa)/\fa^{[p]}$ form a generating set for this
$S$-module; for each $i = 1, \ldots, t$, let $E_i$ denote $(0:_E\fa)
= E_R(R/\fm)$ endowed with the ($x$-torsion-free) left
$R[x,f]$-module structure for which $xe = u_iye$ for all $e \in
(0:_E\fa)$; and let $\fb_i$ be the unique smallest ideal of positive
height in ${\mathcal I}(E_i)$. By \cite[Theorem 3.5]{gatcti}, each
$\fb_i$ is contained in the test ideal $\tau(R)$ of $R$, so that
$\fb := \fb_1 + \cdots + \fb_t \subseteq \tau(R)$. By Proposition
\ref{fp.22}, we have $\fb R_c = R_c$. Hence there exists $h \in \N$
such that $c^h \in \fb \subseteq \tau(R)$, and so there exists $e
\in \nn$ such that $c^{p^e}$ is a test element for $R$.

Therefore, for each finitely generated $R$-module $M$ and each $m
\in 0^*_M$, we have $c^{p^e}x^j(1 \otimes m) = 0$ in
$R[x,f]\otimes_RM$, for all $j \in \nn$. Thus, for $m \in 0^*_M$, we
have
$$
x^ecx^i(1 \otimes m) = c^{p^e}x^{e+i}(1 \otimes m) = 0 \quad
\mbox{for all~} i \in \nn.
$$
However, the left $R[x,f]$-module $R[x,f]\otimes_RM$ is
$x$-torsion-free. Hence $cx^i(1 \otimes m) = 0$ for all $i \in \nn$
and all $m \in 0^*_M$. As this is true for all finitely generated
$R$-modules $M$, it follows that $c$ itself is a test element for
$R$.
\end{proof}

\section{\sc Additional embedding theorems and applications to $F$-pure excellent rings}
\label{aet}

In order to extend certain results from Section \ref{fp} to
non-local $F$-pure rings, we are going to modify some constructions
and arguments from \cite{gatcti}. Suppose that $(R,\fm)$ is local
and $F$-pure, and endow $E := E_R(R/\fm)$ with a structure of
$x$-torsion-free left $R[x,f]$-module (Theorem \ref{fp.1} shows that
this is possible). Let $M$ be a finitely generated $R$-module. In
\cite[Theorem 3.5]{gatcti}, the author showed that the graded left
$R[x,f]$-module $R[x,f]\otimes_RM$ can be embedded, by means of a
homogeneous $R[x,f]$-monomorphism, into a product (in the category
of $\Z$-graded left $R[x,f]$-modules and homogeneous homomorphisms)
of countably many graded left $R[x,f]$-modules, each equal either to
a certain graded left $R[x,f]$-module $\widetilde{E}$ constructed
from $E$ (see \cite[Lemma 2.4]{gatcti}) or to an extension (see
\cite[Definition 2.10]{gatcti}) of a shift of $\widetilde{E}$. In
this paper, we are going to modify those ideas so that we can obtain
a similar embedding for $R[x,f]\otimes_RL$ when $L$ is an arbitrary
$R$-module.

\begin{defi}
\label{pl.3} Let $H$ be a left $R[x,f]$-module, let $a \in R$ and
let $\fb$ be an ideal of $R$. We say that $H$ is {\em
$a$-testable\/} if, whenever $h \in H$ and $c \in R^{\circ}$ are
such that $cx^nh = 0$ for all $n \gg 0$, then $ax^nh = 0$ for all $n
\geq 0$.

It should be noted that there is no requirement that $a \in
R^{\circ}$ in this definition.

Furthermore, we say that $H$ is {\em $\fb$-testable\/} if $H$ is
$r$-testable for all $r \in \fb$. Suppose that $\fb$ can be
generated by $b_1, \ldots, b_w$. Then it is clear that $H$ is
$\fb$-testable if and only if it is $b_i$-testable for all $i = 1,
\ldots, w$.

Many of the uses of this concept will be in tight closure theory, in
situations where the ideal $\fb$ has positive height. In this
context, it should be noted that an ideal $\fb$ of positive height
can be generated by (finitely many) elements of $\fb\cap R^{\circ}$:
see \cite[Lemma 3.4]{gatcti}.
\end{defi}

In this section, we shall make use of some facts about the behaviour
under ring homomorphisms of certain concepts from \cite{gatcti}.

\begin{prop}\label{aet.1} Let $\theta : R \lra R'$ be a homomorphism
of commutative Noetherian rings of characteristic $p$. Then $\theta$
induces a ring homomorphism $\widetilde{\theta} : R[x,f] \lra
R'[x,f]$ for which $\widetilde{\theta}\left(\sum_{i=0}^n
r_ix^i\right) = \sum_{i=0}^n \theta(r_i)x^i$ for all $n \in \nn$,
$r_0, \ldots, r_n \in R$.

Let $H'$ be a left $R'[x,f]$-module. Then $H'$ can be regarded as a
left $R[x,f]$-module by means of $\widetilde{\theta}$, and
\begin{enumerate}
\item $\mathcal{A}_{R[x,f]}(H') \subseteq
\mathcal{A}_{R'[x,f]}(H')$;
\item $\mathcal{G}_{R[x,f]}(H') =
\left\{\widetilde{\theta}^{-1}(\fB') : \fB' \in
\mathcal{G}_{R'[x,f]}(H')\right\}$; and
\item in the special case in which $R' = S^{-1}R$, where $S$ is a
multiplicatively closed subset of $R$ and $\theta$ is the natural
ring homomorphism $\xi : R \lra S^{-1}R$, then
$\mathcal{A}_{R[x,f]}(H') = \mathcal{A}_{(S^{-1}R)[x,f]}(H')$.
\end{enumerate}
\end{prop}

\begin{proof} Let $\phantom{\fa}^e$ stand for extension of ideals
with respect to the ring homomorphism $\theta:R \lra R'$.

(i), (iii) Let $G \in \mathcal{A}_{R[x,f]}(H')$. Thus there exists a
graded two-sided ideal $\fB$ of $R[x,f]$ such that $G =
\ann_{H'}(\fB)$. We can write $\fB = \bigoplus_{i \in \nn}\fb_ix^i$
for a suitable ascending chain $(\fb_i)_{i\in \nn}$ of ideals of
$R$. Thus an element $h' \in H'$ belongs to $G$ if and only if, for
all $i \in \nn$ and all $r_i \in \fb_i$, we have $r_ix^ih' = 0$,
that is, if and only if, for all $i \in \nn$ and all $r_i \in
\fb_i$, we have $\theta(r_i)x^ih' = 0$. Hence
$$
G = \ann_{H'}\left({\textstyle \bigoplus_{i \in
\nn}(\fb_i)^ex^i}\right) \in \mathcal{A}_{R'[x,f]}(H').
$$

Now consider the special case in which $R' = S^{-1}R$ and $\theta$
is the natural ring homomorphism $\xi$. Let $G' \in
\mathcal{A}_{(S^{-1}R)[x,f]}(H')$. Thus there exists a graded
two-sided ideal $\fB'$ of $(S^{-1}R)[x,f]$ such that $G' =
\ann_{H'}(\fB')$. The fact that every ideal of $S^{-1}R$ is extended
from its own contraction to $R$ means that there is an ascending
chain $(\fb_i)_{i\in \nn}$ of ideals of $R$ such that $\fB' =
\bigoplus_{i \in \nn}(\fb_i)^ex^i$, and the argument in the
immediately preceding paragraph shows that
$$
\ann_{H'}\left({\textstyle \bigoplus_{i \in \nn}\fb_ix^i}\right) =
\ann_{H'}\left({\textstyle \bigoplus_{i \in \nn}(\fb_i)^ex^i}\right)
= \ann_{H'}(\fB') = G'.
$$
Therefore $G' \in \mathcal{A}_{R[x,f]}(H')$.

(ii) Let $\fB' \in \mathcal{G}_{R'[x,f]}(H')$, so that $\fB' =
\grann_{R'[x,f]}(G')$ for some $R'[x,f]$-submodule $G'$ of $H'$. Now
$G'$ is automatically an $R[x,f]$-submodule of $H'$, and it is
routine to check that $\grann_{R[x,f]}(G') =
\widetilde{\theta}^{-1}(\fB')$.  Thus we have proved that
$$
\mathcal{G}_{R[x,f]}(H') \supseteq
\left\{\widetilde{\theta}^{-1}(\fB') : \fB' \in
\mathcal{G}_{R'[x,f]}(H')\right\}.
$$

To prove the reverse inclusion, let $\fB \in
\mathcal{G}_{R[x,f]}(H')$. Then, by Lemma \ref{ga1.7(v)}, we have
$\fB = \grann_{R[x,f]}G$, where $G = \ann_{H'}(\fB) \in
\mathcal{A}_{R[x,f]}(H')$. By part (i), $G \in
\mathcal{A}_{R'[x,f]}(H')$, and so $G$ is an $R'[x,f]$-submodule of
$H'$. Hence, by the first paragraph of this proof of part (ii),
$$\fB = \grann_{R[x,f]}G = \widetilde{\theta}^{-1}\left(\grann_{R'[x,f]}G\right) \in
\left\{\widetilde{\theta}^{-1}(\fB') : \fB' \in
\mathcal{G}_{R'[x,f]}(H')\right\}.$$
\end{proof}

\begin{ntn}
\label{fp.3} For the remainder of this section, $I$ will denote a
non-empty, but possibly infinite, indexing set, we shall set $R_i :=
R$ for all $i \in I$, and we shall use $V$ to denote the free
$R$-module $\bigoplus_{i\in I}R_i$. We shall frequently use notation
like `$(r_i)_{i\in I}$' to denote an element of $V$; when we do, it
is of course to be understood that $r_i \neq 0$ for only finitely
many $i \in  I$.

We refer to the mapping $f : V \lra V$ for which $f((r_i)_{i\in I})
= (r_i^p)_{i\in I}$ for all $(r_i)_{i\in I} \in V$ as the {\em
Frobenius map}.
\end{ntn}

\begin{lem}\label{fp.gatcti2.8} Let $b \in \N$ and $W = \bigoplus_{n\geq b}W_n$ be a $\Z$-graded
left $R[x,f]$-module; let $(g_i)_{i\in I}$ be a family of arbitrary
elements of $W_b$ (so that infinitely many of them could be
non-zero), indexed by the set $I$. Use the notation of\/ {\rm
\ref{fp.3}}. Set
$$
K := \left\{ (r_i)_{i\in I} \in V : {\textstyle \sum_{i\in I}r_ig_i
= 0}\right\},
$$
an $R$-submodule of $V$.

Then there is a graded left $R[x,f]$-module
\[
W' = \bigoplus_{n\geq b-1}W'_n = \left(V/f^{-1}(K)\right)\oplus W_b
\oplus W_{b+1} \oplus \cdots \oplus W_t \oplus \cdots
\]
(so that $W'_{b-1} = V/f^{-1}(K)$ and $W'_n = W_n$ for all $n \geq
b$) which has $W$ as an $R[x,f]$-submodule and for which
$x((r_i)_{i\in I} + f^{-1}(K)) = \sum_{i\in I}r_i^pg_i$ for all
$(r_i)_{i\in I} \in V$. We call $W'$ the\/ {\em $1$-place extension
of $W$ by $(g_i)_{i\in I}$}, and denote it by $\ext(W;(g_i)_{i\in
I};1)$.

Let $a \in R$. If $W$ is $a$-testable, then so too is
$\ext(W;(g_i)_{i\in I};1).$

If $W$ is $x$-torsion-free, then so too is $\ext(W;(g_i)_{i\in
I};1)$, and then $$\mathcal{G}(\ext(W;(g_i)_{i\in I};1)) =
\mathcal{G}(W)$$ and $\mathcal{I}(\ext(W;(g_i)_{i\in I};1)) =
\mathcal{I}(W).$
\end{lem}

\begin{proof} All the claims in this Lemma, except the one that
$W' := \ext(W;(g_i)_{i\in I};1)$ is $a$-testable if $W$ is, can be
proved in a manner entirely similar to that used to prove
\cite[Lemma 2.8]{gatcti}.

So suppose that $W$ is $a$-testable. To show that $W'$ is
$a$-testable, it is enough for us to show that, whenever $w$ is a
homogeneous element of $W'$ and $c \in R^{\circ}$ are such that
$cx^nw = 0$ for all $n \gg 0$, then $w$ is annihilated by $Ra
R[x,f]$. This is immediate from the fact that $W$ is $a$-testable
when $w$ has degree at least $b$. We therefore suppose that $w \in
W'_{b-1}$, say $w = (r_i)_{i\in I} + f^{-1}(K)$ for some
$(r_i)_{i\in I} \in V$. Since $cx^n(xw) = 0$ for all $n \gg 0$ and
$xw \in W_b$, it follows that $xw$ is annihilated by $Ra R[x,f]$;
therefore $a x^nw = 0$ for all $n\geq 1$. It remains to show that
$w$ is annihilated by $a$.

Now $xaw = a^pxw = 0$. Therefore
$$
0  = xaw = x((ar_i)_{i\in I} + f^{-1}(K)) = \sum_{i\in
I}a^pr_i^pg_i,
$$
so that $(ar_i)_{i\in I} \in f^{-1}(K)$ and $aw = 0$. Therefore $w$
is annihilated by $a$, and so by $Ra R[x,f]$.
\end{proof}

\begin{defi}\label{fp.gatcti2.10} Let $b \in \N$ and $W = \bigoplus_{n\geq b}W_n$ be a $\Z$-graded
left $R[x,f]$-module, and let $(g_i)_{i\in I}$ be a family of
arbitrary elements of $W_b$, as in Lemma \ref{fp.gatcti2.8}. The
$1$-place extension $\ext(W;(g_i)_{i\in I};1)$ of $W$ by
$(g_i)_{i\in I}$ was defined in Lemma \ref{fp.gatcti2.8}. Recall
that, with the notation of \ref{fp.3}, we defined
$$K := \left\{ (r_i)_{i\in I} \in V : {\textstyle \sum_{i\in I}r_ig_i
= 0}\right\}.
$$

Now let $h \in \N$ with $h \geq 2$. The {\em $h$-place extension
$\ext(W;(g_i)_{i\in I};h)$ of $W$ by $(g_i)_{i\in I}$\/} is the
graded left $R[x,f]$-module
\[ \left(V/f^{-h}(K)\right)\oplus \cdots \oplus
\left(V/f^{-1}(K)\right)\oplus W_b \oplus  \cdots \oplus W_t \oplus
\cdots
\]
which has $\ext(W;(g_i)_{i\in I};1)$ as a graded $R[x,f]$-submodule
and is such that $$x(v+ f^{-j}(K)) = f(v)+ f^{-(j-1)}(K) \quad
\text{for all~} v \in V \text{~and~} j = h,h-1, \ldots, 2.
$$

For each $j \in I$, let $e_j$ denote the element $(r_i)_{i\in I}$ of
$V$ for which $r_j = 1$ and $r_i = 0$ for all $i \in I \setminus
\{j\}$. It is straightforward to check that
$$
\ext(W;(g_i)_{i\in I};h) = \ext(\ext(W;(g_i)_{i\in
I};1);(\overline{e_i})_{i\in I};h-1),$$ where, for $v \in V$, we use
$\overline{v}$ to denote $v + f^{-1}(K)$, and
$$
\ext(W;(g_i)_{i\in I};h) = \ext(\ext(W;(g_i)_{i\in
I};h-1);(\widetilde{e_i})_{i\in I};1)$$ where, for $v \in V$, we use
$\widetilde{v}$ to denote $v + f^{-(h-1)}(K)$.

Let $a \in R$. It is a consequence of Lemma \ref{fp.gatcti2.8} that,
if $W$ is $a$-testable, then so to is $\ext(W;(g_i)_{i\in I};h)$.

It is also a consequence of Lemma \ref{fp.gatcti2.8} that, if $W$ is
$x$-torsion-free, then so too is $\ext(W;(g_i)_{i\in I};h)$, and
then $\mathcal{G}(\ext(W;(g_i)_{i\in I};h)) = \mathcal{G}(W)$ and
$$\mathcal{I}(\ext(W;(g_i)_{i\in I};h)) = \mathcal{I}(W).$$

Finally, it will occasionally be convenient to extend the
terminology and regard $W$ itself as a $0$-place extension of $W$.
\end{defi}

\begin{prop}\label{fp.gatcti2.11} Let $b \in \N$ and $W = \bigoplus_{n\geq b}W_n$ be a $\Z$-graded
left $R[x,f]$-module. Use the notation of\/ {\rm \ref{fp.3}}. Let
$\{m_i : i \in I \}$ be a generating set for an $R$-module $M$. We
can form the graded $R[x,f]$-submodule $\bigoplus_{n \geq
b}(Rx^n\otimes_RM)$ of $R[x,f]\otimes_RM$. Suppose that there is
given a homogeneous $R[x,f]$-homomorphism $\lambda' =
\bigoplus_{n\geq b}\lambda_n : \bigoplus_{n \geq b}(Rx^n\otimes_RM)
\lra W$.

For each $i \in I$, let $g_i := \lambda_b(x^b\otimes m_i) \in W_b$.
Set
$$K := \left\{ (r_i)_{i\in I} \in V : {\textstyle \sum_{i\in I}r_ig_i
= 0}\right\},
$$ as in Lemma\/ {\rm \ref{fp.gatcti2.8}}. For each $n = 0, 1, \ldots, b-1$,
there exists an $R$-homomorphism $\lambda_n : Rx^n \otimes_R M \lra
V/f^{-(b-n)}(K)$ such that
$$
\lambda_n\left({\textstyle \sum_{i\in I} r_ix^n\otimes m_i}\right) =
(r_i)_{i\in I} + f^{-(b-n)}(K) \quad \text{for all~} (r_i)_{i\in I}
\in V.
$$
Furthermore, $$ \lambda := \bigoplus_{n\in \nn}\lambda_n :
R[x,f]\otimes_RM = \bigoplus_{n\in \nn}(Rx^n\otimes_RM) \lra
\ext(W;(g_i)_{i\in I};b)
$$
is a homogeneous $R[x,f]$-homomorphism that extends $\lambda'$.
\end{prop}

\begin{proof} Let $n \in \{0,\ldots,b-1\}$. Note that the $R$-module
$Rx^n\otimes_RM$ is generated by $\{x^n\otimes m_i : i \in I\}$. Now
if $(r_i)_{i\in I}, (s_i)_{i\in I} \in V$ are such that $\sum_{i\in
I}r_ix^n\otimes m_i = \sum_{i\in I}s_ix^n\otimes m_i$, then
multiplication of both sides on the left by $x^{b-n}$ yields that
$\sum_{i\in I}r_i^{p^{b-n}}x^b\otimes m_i = \sum_{i\in
I}s_i^{p^{b-n}}x^b\otimes m_i$. Apply the $R$-homomorphism
$\lambda_b$ to deduce that
\begin{align*}
\sum_{i\in I}(r_i - s_i)^{p^{b-n}}g_i & = \sum_{i\in
I}(r_i^{p^{b-n}} - s_i^{p^{b-n}})\lambda_b(x^b\otimes m_i)\\ &=
\lambda_b \left(\sum_{i\in I}r_i^{p^{b-n}}x^b\otimes m_i -
\sum_{i\in I}s_i^{p^{b-n}}x^b\otimes m_i\right) = \lambda_b(0 ) = 0.
\end{align*}
Therefore $(r_i - s_i)_{i\in I} \in f^{-(b-n)}(K)$, and $(r_i)_{i\in
I} + f^{-(b-n)}(K) = (s_i)_{i\in I} + f^{-(b-n)}(K)$. After this,
all the remaining claims are easy to check.
\end{proof}

\begin{rmds}
\label{fp.7} We shall need to use two constructions from
\cite[\S2]{gatcti}, and we include here reminders for the reader's
convenience.

\begin{enumerate} \item Let
$\left(H^{(\lambda)}\right)_{\lambda\in\Lambda}$ be a non-empty
family of $\Z$-graded left $R[x,f]$-modules, with gradings given by
$H^{(\lambda)}= \bigoplus_{n\in\Z}H^{(\lambda)}_n$ for each
$\lambda\in\Lambda$. The $R$-module
\[
\prod_{\lambda\in\Lambda}{\textstyle ^{^{^{\Large
\prime}}}}H^{(\lambda)} :=
\bigoplus_{n\in\Z}\left(\prod_{\lambda\in\Lambda}H^{(\lambda)}_n\right)
\]
is a ($\Z$-graded) left $R[x,f]$-module in which
\[x\big(h^{(\lambda)}_n\big)_{\lambda\in\Lambda} =
\big(xh^{(\lambda)}_n\big)_{\lambda\in\Lambda}\in
\prod_{\lambda\in\Lambda}H^{(\lambda)}_{n+1} \quad \mbox{for all~}
\big(h^{(\lambda)}_n\big)_{\lambda\in\Lambda} \in
\prod_{\lambda\in\Lambda}H^{(\lambda)}_n.\] In this paper, we shall
refer to $\prod^{\prime}_{\lambda\in\Lambda}H^{(\lambda)}$ as the
{\em graded product\/} of the $H^{(\lambda)}$.

Let $a \in R$. It is clear that if $H^{(\lambda)}$ is $a$-testable
for all $\lambda\in\Lambda$, then
$\prod^{\prime}_{\lambda\in\Lambda}H^{(\lambda)}$ is also
$a$-testable.
\item Let $H$ be a left $R[x,f]$-module.
For all $n \in \nn$, set $H_n := H$. Then the $R$-module
$\widetilde{H} := \bigoplus_{n\in\nn} H_n$ has a natural structure
as a graded left $R[x,f]$-module under which the result of
multiplying $h_n \in H_n = H$ on the left by $x$ is the element
$xh_n \in H_{n+1} = H$. In this paper, we shall refer to
$\widetilde{H}$ as the {\em graded companion\/} of $H$.

Let $a \in R$. It is clear that if $H$ is $a$-testable, then so too
is $\widetilde{H}$.
\item Let $\left(L^{(\lambda)}\right)_{\lambda\in\Lambda}$ be a non-empty
family of (ungraded) left $R[x,f]$-modules. Then it is easy to check
that
$$
\prod_{\lambda\in\Lambda}{\textstyle ^{^{^{\Large
\prime}}}}\widetilde{L^{(\lambda)}} = \widetilde{
\prod_{\lambda\in\Lambda}L^{(\lambda)}}.
$$
Thus, speaking loosely, the operations of taking products and graded
companions commute.
\end{enumerate}
\end{rmds}

\begin{lem}
\label{fp.7z} Let $G := \bigoplus_{n\in\nn} G_n$ be a positively
graded $x$-torsion-free left $R[x,f]$-module for which ${\mathcal
I}(G)$ is a finite set, and let $\fb$ be the smallest ideal of
positive height in ${\mathcal I}(G)$. Then any graded product $K$ of
extensions of shifts of graded products of copies of $G$ is again
$x$-torsion-free, has ${\mathcal I}(K) = {\mathcal I}(G)$, and is
$\fb$-testable.

For example, if $E$ is an $x$-torsion-free left $R[x,f]$-module for
which ${\mathcal I}(E)$ is a finite set, and $\fd$ denotes the
smallest ideal of positive height in ${\mathcal I}(E)$, then any
graded product $L$ of extensions of shifts of graded products of
copies of the graded companion $\widetilde{E}$ of $E$ is again
$x$-torsion-free, has ${\mathcal I}(L) = {\mathcal I}(E)$, and is
$\fd$-testable.
\end{lem}

\begin{proof} It follows from Lemma \ref{fp.gatcti2.8} that the
process of extension preserves torsion-freeness and the set of
special $R$-ideals; it is clear that shifting also preserves
torsion-freeness and the set of special $R$-ideals; and \cite[Lemma
2.3]{gatcti} shows that the graded product, say $K$, of a non-empty
family of $x$-torsion-free $\Z$-graded left $R[x,f]$-modules which
all have the same set ${\mathcal I}$ of special $R$-ideals is again
$x$-torsion-free with ${\mathcal I}(K) = {\mathcal I}$. If we denote
by $\fb$ the smallest ideal of positive height in ${\mathcal I}$,
then it follows from \cite[Theorem 3.12]{ga} that $K$ is
$\fb$-testable.

Finally, the process of passing from $E$ to its graded companion
$\widetilde{E}$ preserves torsion-freeness and the set of special
$R$-ideals, by \cite[Lemma 2.5]{gatcti}, and so the claims in the
final paragraph of the lemma follow on application of the first
paragraph to $\widetilde{E}$.
\end{proof}

The next lemma is a generalization of Lemma 3.1 of \cite{gatcti}.

\begin{lem}\label{fp.gatcit3.1} Suppose that $(R,\fm)$ is local and that there exists an $\nn$-graded left
$R[x,f]$-module $G = \bigoplus_{n\in\nn}G_n$ such that $G_0$ is
$R$-isomorphic to $E_R(R/\fm)$, the injective envelope of the simple
$R$-module $R/\fm$.

Let $M$ be an $\fm$-torsion $R$-module. Then the injective envelope
$E_R(M)$ of $M$ is $R$-isomorphic to a direct sum of copies of
$E_R(R/\fm)$, say to $\bigoplus_{j\in J}E^{(j)}$, where $E^{(j)} =
G_0$ for all $j$ in an indexing set $J$. Set $G^{(j)} = G$ for all
$j \in J$. Then there exists a homogeneous $R[x,f]$-homomorphism
\[
\lambda := \bigoplus_{n\in \nn}\lambda_n : R[x,f]\otimes_RM =
\bigoplus_{n\in \nn}(Rx^n\otimes_RM) \lra \prod_{j\in J}{\textstyle
^{^{^{\Large \prime}}}}G^{(j)},
\]
such that $\lambda_0$ is a monomorphism.
\end{lem}

\begin{proof} The fact that $E_R(M)$ is $R$-isomorphic to a direct sum of copies of
$E_R(R/\fm)$ follows from E. Matlis's theory of injective modules
over a commutative Noetherian ring, described in (for example)
\cite[Chapter 18]{HM}. Set
$$
L = \bigoplus_{n\in\nn}L_n := \prod_{j\in J}{\textstyle ^{^{^{\Large
\prime}}}}G^{(j)}.
$$

Since $\bigoplus_{j\in J}G^{(j)}_0$ can be embedded in $\prod_{j\in
J}G^{(j)}_0$, there is an $R$-monomorphism $\lambda_0 : M \lra L_0$.
We can define, for each $n \in \nn$, an $R$-homomorphism
$$\lambda_n: Rx^n \otimes_RM \lra L_n$$ such that $\lambda_n(rx^n
\otimes m) = rx^n\lambda_0(m)$ for all $r \in R$ and all $m \in M$.
It is straightforward to check that the $\lambda_n~(n\in\nn)$
provide a homogeneous $R[x,f]$-homomorphism as claimed.
\end{proof}

We are now able to prove an embedding theorem, for $\fm$-torsion
modules over the local ring $(R,\fm)$, that is reminiscent of the
result in the first part of \cite[Theorem 3.5]{gatcti}.

\begin{thm}\label{fp.gatcti3.5}
Suppose that $(R,\fm)$ is local, and that there exists an
$\nn$-graded left $R[x,f]$-module $G = \bigoplus_{n\in\nn}G_n$ such
that $G_0$ is $R$-isomorphic to $E_R(R/\fm)$, the injective envelope
of the simple $R$-module $R/\fm$. Let $M$ be an $\fm$-torsion
$R$-module. Then there is a family $\left(H^{(n)}\right)_{n \in
\nn}$ of\/ $\nn$-graded left $R[x,f]$-modules, where $H^{(n)}$ is an
$n$-place extension of a shift of a graded product of copies of $G$
(for each $n \in \nn$), for which there exists a homogeneous
$R[x,f]$-monomorphism
\[
\nu : R[x,f]\otimes_RM = \bigoplus_{i\in \nn}(Rx^i\otimes_RM) \lra
\prod_{n\in\nn}{\textstyle ^{^{^{\!\!\Large \prime}}}} H^{(n)} =: K.
\]

If $G$ is $x$-torsion-free and ${\mathcal I}(G)$ is finite, and we
let $\fb$ be the smallest ideal of positive height in ${\mathcal
I}(G)$, then $K$ is $x$-torsion-free with ${\mathcal I}(K) =
{\mathcal I}(G)$ and $R[x,f]\otimes_RM$ is $\fb$-testable.
\end{thm}

\begin{proof} For each $n \in \nn$, the (left) $R$-module $Rx^n
\otimes_RM$ is $\fm$-torsion, since, if $\fm^tg = 0$ for a $g \in M$
and some $t \in \N$, then $(\fm^t)^{[p^n]}(x^n \otimes g) = 0$.

Let $n \in \nn$. By Lemma \ref{fp.gatcit3.1}, there is a family
$(G^{(j,n)})_{j \in Y_n}$ of graded left $R[x,f]$-modules, all equal
to $G$, and a homogeneous $R[x,f]$-homomorphism
$$R[x,f]\otimes_R(Rx^n\otimes_R M) \lra \prod_{j\in
Y_n}{\textstyle ^{^{^{\Large \prime}}}}G^{(j,n)}$$ which is
monomorphic in degree $0$. If we now use isomorphisms of the type
described in \cite[Remark 3.2]{gatcti}, we obtain (after application
of the shift functor $ (\: {\scriptscriptstyle \bullet} \:)(-n)$) a
homogeneous $R[x,f]$-homomorphism
$$
\zeta^{(n)} : \bigoplus_{j \geq n}(Rx^j\otimes_RM) \lra
\left(\prod_{j\in Y_n}{\textstyle ^{^{^{\Large
\prime}}}}G^{(j,n)}\right)(-n)
$$
which is monomorphic in degree $n$. We can now use Proposition
\ref{fp.gatcti2.11} to extend $\zeta^{(n)}$ by $n$ places to produce
a homogeneous $R[x,f]$-homomorphism
$$
\lambda^{(n)} : \bigoplus_{j \geq 0}(Rx^j\otimes_RM) =
R[x,f]\otimes_RM \lra H^{(n)},
$$
where $H^{(n)}$ is an appropriate $n$-place extension of
$\left(\prod_{j\in Y_n}^{\prime} G^{(j,n)}\right)(-n)$, such that
$\lambda^{(n)}$ is monomorphic in degree $n$.

There is therefore a homogeneous $R[x,f]$-homomorphism
\[
\nu  = \bigoplus_{j\in \nn}\nu_j : R[x,f]\otimes_RM \lra \prod_{
n\in\nn}{\textstyle ^{^{^{\!\!\Large \prime}}}} H^{(n)} =: K
\]
such that $\nu_j(\xi_j) =
\big((\lambda^{(n)})_j(\xi_j)\big)_{n\in\nn}$ for all $j \in \nn$
and $\xi_j \in Rx^j \otimes_RM$. For each $j \in \nn$, we have that
$(\lambda^{(j)})_j$ is a monomorphism; hence $\nu_j$ is a
monomorphism. Hence $\nu$ is an $R[x,f]$-monomorphism.

Lemma \ref{fp.7z} shows that $K$ is $x$-torsion-free with ${\mathcal
I}(K) = {\mathcal I}(G)$, and that $K$ is $\fb$-testable. In view of
the $R[x,f]$-monomorphism $\nu$, it follows that $R[x,f]\otimes_RM$
is $\fb$-testable.
\end{proof}

\begin{cor}\label{fp.9z} Suppose that $(R,\fm)$ is $F$-pure and
local. Then the left $R[x,f]$-module $R[x,f]\otimes_RE_R(R/\fm)$ is
$x$-torsion-free, and its set ${\mathcal
I}(R[x,f]\otimes_RE_R(R/\fm))$ of
$(R[x,f]\otimes_RE_R(R/\fm))$-special $R$-ideals is finite.

In fact, for any $x$-torsion-free left $R[x,f]$-module structure on
$E_R(R/\fm)$ that extends its $R$-module structure (and such exist,
by Theorem\/ {\rm \ref{fp.1}}), we have $${\mathcal
I}(R[x,f]\otimes_RE_R(R/\fm)) \subseteq {\mathcal I}(E_R(R/\fm)),$$
and the latter set is finite; furthermore, if $\fb$ is the smallest
ideal in ${\mathcal I}(E_R(R/\fm))$ of positive height, then
$R[x,f]\otimes_RE_R(R/\fm)$ is $\fb$-testable.
\end{cor}

\begin{proof} In view of Theorem \ref{fp.1}, it is enough to prove the
claim in the second paragraph. So select an $x$-torsion-free left
$R[x,f]$-module structure on $E := E_R(R/\fm)$ that extends its
$R$-module structure.  Construct the graded companion
$\widetilde{E}$ of $E$, as in Reminders \ref{fp.7}(ii), and recall
that $\widetilde{E}$ is again $x$-torsion-free and has ${\mathcal
I}(\widetilde{E}) = {\mathcal I}(E)$. Since $E$ is $\fm$-torsion, it
follows from Theorem \ref{fp.gatcti3.5} that there exists an
$x$-torsion-free $\nn$-graded left $R[x,f]$-module $K$ with
${\mathcal I}(K) = {\mathcal I}(E)$, and a homogeneous
$R[x,f]$-monomorphism $ \nu : R[x,f]\otimes_RE = \bigoplus_{i\in
\nn}(Rx^i\otimes_RE) \lra K$. Therefore $R[x,f]\otimes_RE$ is
$x$-torsion-free and ${\mathcal I}(R[x,f]\otimes_RE) \subseteq
{\mathcal I}(E).$ It follows from \cite[Corollary 3.11]{ga} that
${\mathcal I}(E)$ is finite; we can then use \cite[Theorem 3.12]{ga}
to deduce that $K$ and $R[x,f]\otimes_RE$ are $\fb$-testable.
\end{proof}

\begin{cor}\label{fp.9y} Suppose that $(R,\fm)$ is $F$-pure, local and
complete. Let $c \in R^{\circ}$ be such that $R_c$ is regular. Let
$M$ be an $\fm$-torsion $R$-module. Then the left $R[x,f]$-module
$R[x,f]\otimes_RM$ is $x$-torsion-free and $c$-testable.
\end{cor}

\begin{proof} We appeal to Cohen's Structure Theorem for complete local
rings containing a subfield and write $R = S/\fa$, where $S$ is a
complete regular local ring of characteristic $p$ and $\fa$ is a
proper, non-zero ideal of $S$. This is consistent with the notation
of \ref{hrlr.1}, and we shall use that notation for the remainder of
this proof.

We now appeal to Proposition \ref{fp.22}. Use Lemma\/ {\rm
\ref{fp.23}} to find $u_1, \ldots, u_t \in (\fa^{[p]}:\fa) \setminus
\fn^{[p]}$ such that their natural images in
$(\fa^{[p]}:\fa)/\fa^{[p]}$ form a generating set for this
$S$-module; for each $i = 1, \ldots, t$, let $E_i$ denote $(0:_E\fa)
= E_R(R/\fm)$ endowed with the ($x$-torsion-free) left
$R[x,f]$-module structure for which $xe = u_iye$ for all $e \in
(0:_E\fa)$; and let $\fb_i$ be the unique smallest ideal of positive
height in ${\mathcal I}(E_i)$.

By Theorem \ref{fp.gatcti3.5} (applied to the graded companions
$\widetilde{E_1}, \ldots, \widetilde{E_t}$ in turn), we deduce that
$R[x,f]\otimes_RM$ is $x$-torsion-free, and $\fb_i$-testable for all
$i = 1, \ldots, t$. Therefore $R[x,f]\otimes_RM$ is $(\fb_1 + \cdots
+ \fb_t)$-testable. By Proposition \ref{fp.22}, we have $(\fb_1 +
\cdots + \fb_t) R_c = R_c$. Hence some power of $c$ belongs to
$\fb_1 + \cdots + \fb_t$, and there exists $e \in \nn$ such that
$R[x,f]\otimes_RM$ is $c^{p^e}$-testable.

Therefore, for each $y \in R[x,f]\otimes_RM$ for which there exists
$d \in R^{\circ}$ such that $dx^ny = 0$ for all $n \gg 0$, we have
$c^{p^e}x^jy = 0$ for all $j \in \nn$. Thus
$$
x^ecx^iy = c^{p^e}x^{e+i}y = 0 \quad \mbox{for all~} i \in \nn.
$$
However, $R[x,f]\otimes_RM$ is $x$-torsion-free, and so $cx^iy = 0$
for all $i \in \nn$. Therefore $R[x,f]\otimes_RM$ is $c$-testable.
\end{proof}

\begin{cor}\label{fp.9x} Suppose that $(R,\fm)$ is $F$-pure,
local and excellent. Let $c \in R^{\circ}$ be such that $R_c$ is
regular. Let $M$ be an $\fm$-torsion $R$-module. Then the left
$R[x,f]$-module $R[x,f]\otimes_RM$ is $x$-torsion-free and
$c$-testable.

In particular, $R[x,f]\otimes_RE_R(R/\fm)$ is $x$-torsion-free and
$c$-testable.
\end{cor}

\begin{proof} Note that $\widehat{R}$ is again $F$-pure, by Lemma \ref{mr.HR}. Complete local rings
are always excellent. Furthermore, the fibre rings of the flat ring
homomorphism $R_c \lra \widehat{R}_c$ induced by inclusion are rings
of fractions of the formal fibres of $R$, and so are regular; it
therefore follows that $\widehat{R}_c$ is regular. Note that $c \in
\widehat{R}^{\circ}$.

We now appeal to Corollary \ref{fp.9y} to deduce that the graded
left $\widehat{R}[x,f]$-module
$$
\widehat{R}[x,f]\otimes_{\widehat{R}}E_{\widehat{R}}(\widehat{R}/\fm\widehat{R})
$$
is $x$-torsion-free and $c$-testable. Regard this as an
$R[x,f]$-module in the natural way: it is still $x$-torsion-free and
$c$-testable, and its $0$th component is $R$-isomorphic to
$E_R(R/\fm)$.

By Theorem \ref{fp.gatcti3.5}, there is a family
$\left(H^{(n)}\right)_{n \in \nn}$ of\/ $\nn$-graded left
$R[x,f]$-modules, where $H^{(n)}$ is an $n$-place extension of a
shift of a graded product of copies of
$$\widehat{R}[x,f]\otimes_{\widehat{R}}E_{\widehat{R}}(\widehat{R}/\fm\widehat{R})$$
(for each $n \in \nn$), for which there exists a homogeneous
$R[x,f]$-monomorphism
\[
\nu : R[x,f]\otimes_RM = \bigoplus_{i\in \nn}(Rx^i\otimes_RM) \lra
\prod_{n\in\nn}{\textstyle ^{^{^{\!\!\Large \prime}}}} H^{(n)} =: K.
\]
By Lemma \ref{fp.gatcti2.8} and Reminders \ref{fp.7}(i), the left
$R[x,f]$-module $K$ is $x$-torsion-free and $c$-testable, and
therefore so too is $R[x,f]\otimes_RM$.

The final claim now follows from the fact that $E_R(R/\fm)$ is an
$\fm$-torsion $R$-module.
\end{proof}

\begin{cor}\label{fp.9w} Suppose that $R$ is $F$-pure and excellent (but
not necessarily local). Let $c \in R^{\circ}$ be such that $R_c$ is
regular.

Let $\fp \in \Spec (R)$, and let $k(\fp)$ denote the simple
$R_{\fp}$-module. View the left $R_\fp[x,f]$-module
$$R_\fp[x,f]\otimes_{R_{\fp}}E_{R_{\fp}}(k(\fp))$$ as an
$R[x,f]$-module in the natural way. Then
$R_\fp[x,f]\otimes_{R_{\fp}}E_{R_{\fp}}(k(\fp))$ is $x$-torsion-free
and $c$-testable.
\end{cor}

\begin{proof} By Lemma \ref{mr.HR}, the localization $R_{\fp}$ is again $F$-pure; of
course, $R_{\fp}$ is excellent. Furthermore, since
$\left(R_{\fp}\right)_{c/1}$, the ring of fractions of $R_{\fp}$
with respect to the set of powers of the element $c/1$ of
$(R_{\fp})^{\circ}$, is a ring of fractions of $R_c$, it is regular.

We can now appeal to Corollary \ref{fp.9x} to deduce that the left
$R_\fp[x,f]$-module $R_\fp[x,f]\otimes_{R_{\fp}}E_{R_{\fp}}(k(\fp))$
is $x$-torsion-free and $(c/1)$-testable. Therefore, when we regard
$R_\fp[x,f]\otimes_{R_{\fp}}E_{R_{\fp}}(k(\fp))$ as an
$R[x,f]$-module in the natural way, it is $x$-torsion-free and
$c$-testable.
\end{proof}

\begin{lem}
\label{fp.31} For each $\fp \in \Spec(R)$, let $k(\fp)$ denote the
simple $R_{\fp}$-module and let
$$
H(\fp) := R_\fp[x,f]\otimes_{R_{\fp}}E_{R_{\fp}}(k(\fp)),
$$
viewed as a graded left $R[x,f]$-module in the natural way. Observe
that the $0$th component of $H(\fp)$ is $R$-isomorphic to
$E_R(R/\fp)$.

Let $M$ be an $R$-module, and express its injective envelope
$E_R(M)$ as a direct sum of indecomposable injective $R$-modules,
say $E_R(M) \cong \bigoplus_{j\in J}E_R(R/\fp^{(j)})$, where
$\left(\fp^{(j)}\right)_{j\in J}$ is an appropriate family of prime
ideals of $R$. (The superscript$\phantom{\fp}^{(j)}$ here is an
index and does not indicate symbolic power.) Then there exists a
homogeneous $R[x,f]$-homomorphism
\[
\lambda := \bigoplus_{n\in \nn}\lambda_n : R[x,f]\otimes_RM =
\bigoplus_{n\in \nn}(Rx^n\otimes_RM) \lra \prod_{j\in J}{\textstyle
^{^{^{\Large \prime}}}}H(\fp^{(j)}),
\]
such that $\lambda_0$ is a monomorphism.
\end{lem}

\begin{proof} Set
$$
L = \bigoplus_{n\in\nn}L_n := \prod_{j\in J}{\textstyle ^{^{^{\Large
\prime}}}}H(\fp^{(j)}).
$$
Since $\bigoplus_{j\in J}H(\fp^{(j)})_0 \cong \bigoplus_{j\in
J}E_R(R/\fp^{(j)})$ and $\bigoplus_{j\in J}H(\fp^{(j)})_0$ can be
embedded in $\prod_{j\in J}H(\fp^{(j)})_0$, there is an
$R$-monomorphism $\lambda_0 : M \lra L_0$. We can define, for each
$n \in \nn$, an $R$-homomorphism $\lambda_n: Rx^n \otimes_RM \lra
L_n$ such that $\lambda_n(rx^n \otimes m) = rx^n\lambda_0(m)$ for
all $r \in R$ and all $m \in M$. It is straightforward to check that
the $\lambda_n~(n\in\nn)$ provide a homogeneous
$R[x,f]$-homomorphism as claimed.
\end{proof}

We are now in a position to prove the main result of this paper.

\begin{thm}
\label{fp.32} For each $\fp \in \Spec(R)$, let $k(\fp)$ denote the
simple $R_{\fp}$-module and let
$$
H(\fp) := R_\fp[x,f]\otimes_{R_{\fp}}E_{R_{\fp}}(k(\fp)),
$$
viewed as a graded left $R[x,f]$-module in the natural way.

Let $M$ be an $R$-module. Then there is a family
$\left(G^{(n)}\right)_{n \in \nn}$ of\/ $\nn$-graded left
$R[x,f]$-modules, where $G^{(n)}$ is an $n$-place extension of a
shift of a graded product of graded left $R[x,f]$-modules of the
form $H(\fp)$ for various prime ideals $\fp$ of $R$ (for each $n \in
\nn$), for which there exists a homogeneous $R[x,f]$-monomorphism
\[
\nu : R[x,f]\otimes_RM = \bigoplus_{i\in \nn}(Rx^i\otimes_RM) \lra
\prod_{n\in\nn}{\textstyle ^{^{^{\!\!\Large \prime}}}} G^{(n)}.
\]

If $R$ is excellent and $F$-pure, and $c \in R^{\circ}$ is such that
$R_c$ is regular, then $R[x,f]\otimes_RM$ is $c$-testable; as this
is true for each $R$-module $M$, it follows that $c$ is a big test
element for $R$.

Thus if $R$ is excellent and $F$-pure, then $R$ has a big test
element.
\end{thm}

\begin{proof} Let $n \in \nn$. By Lemma \ref{fp.31}, there is a
family $\left(\fp^{(j,n)}\right)_{j \in Y_n}$ of prime ideals of
$R$, and a homogeneous $R[x,f]$-homomorphism
$$R[x,f]\otimes_R(Rx^n\otimes_R M) \lra \prod_{j\in
Y_n}{\textstyle ^{^{^{\Large \prime}}}}H(\fp^{(j,n)})$$ which is
monomorphic in degree $0$. If we now use isomorphisms of the type
described in \cite[Remark 3.2]{gatcti}, we obtain (after application
of the shift functor $ (\: {\scriptscriptstyle \bullet} \:)(-n)$) a
homogeneous $R[x,f]$-homomorphism
$$
\zeta^{(n)} : \bigoplus_{j \geq n}(Rx^j\otimes_RM) \lra
\left(\prod_{j\in Y_n}{\textstyle ^{^{^{\Large
\prime}}}}H(\fp^{(j,n)})\right)(-n)
$$
which is monomorphic in degree $n$. We can now use Proposition
\ref{fp.gatcti2.11} to extend $\zeta^{(n)}$ by $n$ places to produce
a homogeneous $R[x,f]$-homomorphism
$$
\lambda^{(n)} : \bigoplus_{j \geq 0}(Rx^j\otimes_RM) =
R[x,f]\otimes_RM \lra G^{(n)},
$$
where $G^{(n)}$ is an appropriate $n$-place extension of
$\left(\prod_{j\in Y_n}^{\prime} H(\fp^{(j,n)})\right)(-n)$, such
that $\lambda^{(n)}$ is monomorphic in degree $n$.

There is therefore a homogeneous $R[x,f]$-homomorphism
\[
\nu  = \bigoplus_{j\in \nn}\nu_j : R[x,f]\otimes_RM = \bigoplus_{j
\in \nn}(Rx^j\otimes_RM) \lra \prod_{ n\in\nn}{\textstyle
^{^{^{\!\!\Large \prime}}}} G^{(n)} =: K
\]
such that $\nu_j(\xi_j) =
\big((\lambda^{(n)})_j(\xi_j)\big)_{n\in\nn}$ for all $j \in \nn$
and $\xi_j \in Rx^j \otimes_RM$. For each $j \in \nn$, we have that
$(\lambda^{(j)})_j$ is a monomorphism; hence $\nu_j$ is a
monomorphism. Hence $\nu$ is an $R[x,f]$-monomorphism.

Now suppose that $R$ is excellent and $F$-pure, and that $c \in
R^{\circ}$ is such that $R_c$ is regular. By Corollary \ref{fp.9w},
the left $R[x,f]$-module $H(\fp)$ is $x$-torsion-free and
$c$-testable.  It now follows from Lemma \ref{fp.gatcti2.8} and
Reminders \ref{fp.7} that $\prod^{\prime}_{j\in Y_n}H(\fp^{(j,n)})$
is $x$-torsion-free and $c$-testable for all $n \in \nn$, that
$G^{(n)}$ is $x$-torsion-free and $c$-testable for all $n \in \nn$,
and that $K$ is $x$-torsion-free and $c$-testable.

Therefore $R[x,f]\otimes_RM$ is ($x$-torsion-free and) $c$-testable
for each $R$-module $M$, and so $c$ is a big test element for $R$.
\end{proof}

Experts in tight closure theory will know that it is desirable to
have `locally stable' or `completely stable' test elements.  These
are defined as follows.

\begin{defs}
\label{fp.33} A test element for modules (respectively, for ideals)
$c$ for $R$ is said to be a {\em locally stable test element for
modules (respectively, for ideals) for $R$\/} if and only if, for
every $\fp \in \Spec (R)$, the natural image $c/1$ of $c$ in
$R_{\fp}$ is a test element for modules (respectively, for ideals)
for $R_{\fp}$.

A test element for modules (respectively, for ideals) $c$ for $R$ is
said to be a {\em completely stable test element for modules
(respectively, for ideals) for $R$\/} if and only if, for every $\fp
\in \Spec (R)$, the natural image $c/1$ of $c$ in
$\widehat{R_{\fp}}$ is a test element for modules (respectively, for
ideals) for $\widehat{R_{\fp}}$.
\end{defs}

We make corresponding definitions for big test elements.

\begin{defs}
\label{fp.34} A big test element $c$ for $R$ is said to be a {\em
locally stable (respectively, completely stable) big test element
for $R$} if and only if, for every $\fp \in \Spec (R)$, the natural
image $c/1$ of $c$ in $R_{\fp}$ (respectively, $\widehat{R_{\fp}}$)
is a big test element for $R_{\fp}$ (respectively, for
$\widehat{R_{\fp}}$).
\end{defs}

We can now use Theorem \ref{fp.32} to obtain big test elements that
are locally stable and completely stable.

\begin{cor}
\label{fp.35} Let $R$ be excellent and $F$-pure, and let $c \in
R^{\circ}$ be such that $R_c$ is regular. Then $c$ is a locally
stable big test element for $R$ and a completely stable big test
element for $R$.
\end{cor}

\begin{proof} We proved in Theorem \ref{fp.32} that $c$ is a big
test element for $R$. Let $\fp \in \Spec (R)$. Of course, $c/1 \in
R_{\fp}$ belongs to both $\big(R_{\fp}\big)^{\circ}$ and
$\big(\widehat{R_{\fp}}\big)^{\circ}$.

By Lemma \ref{mr.HR}, the localization $R_{\fp}$ and its completion
$\widehat{R_{\fp}}$ are again $F$-pure; of course, $R_{\fp}$ and
$\widehat{R_{\fp}}$ are excellent. Furthermore, since
$\big(R_{\fp}\big)_{c/1}$, the ring of fractions of $R_{\fp}$ with
respect to the set of powers of the element $c/1$ of
$(R_{\fp})^{\circ}$, is a ring of fractions of $R_c$, it is regular.
In addition, the fibre rings of the flat ring homomorphism
$\big(R_{\fp}\big)_{c/1} \lra \big(\widehat{R_{\fp}}\big)_{c/1}$
induced by inclusion are rings of fractions of the formal fibres of
$R_{\fp}$, and so are regular; it therefore follows that
$\big(\widehat{R_{\fp}}\big)_{c/1}$ is regular.

We can therefore use Theorem \ref{fp.32} to deduce that $c/1$ is a
big test element for both $R_{\fp}$ and $\widehat{R_{\fp}}$.
\end{proof}


\begin{thebibliography}{00}

\bibitem{Blickle}M. Blickle, {\it The intersection homology
 $D$-module in finite characteristic\/}, PhD dissertation, University
 of Michigan, Ann Arbor, 2001.

\bibitem{LC}
 M. P. Brodmann and R. Y. Sharp, {\it Local cohomology: an algebraic
 introduction with geometric applications\/},
 Cambridge Studies in
 Advanced Mathematics {\bf 60}, Cambridge University Press, 1998.

\bibitem{Fedde83}R. Fedder, {\it $F$-purity and rational singularity\/},
 Transactions Amer.\ Math.\ Soc.\ {\bf 278} (1983)
 461--480.

\bibitem{HarSpe77}R. Hartshorne and R. Speiser, {\it Local cohomological
 dimension in characteristic $p$\/}, Annals of Math. {\bf 105} (1977) 45--79.

\bibitem{HocHun90}M. Hochster and C. Huneke, {\it Tight closure,
 invariant theory and the Brian\c{c}on-Skoda Theorem\/}, J. Amer.\ Math.\
 Soc.\ {\bf 3} (1990) 31--116.

\bibitem{HocHun94}M. Hochster and C. Huneke,
 {\it $F$-regularity, test elements, and smooth base change\/},
 Transactions Amer.\ Math.\ Soc.\ {\bf 346} (1994)
 1--62.

\bibitem{HocRob74}M. Hochster and J. L. Roberts, {\it Rings of invariants of
 reductive groups acting on regular rings are Cohen--Macaulay\/},
 Advances in Math. {\bf 13} (1974) 115--175.

\bibitem{Hunek96}C. Huneke, \textit{Tight closure and its applications\/},
 Conference Board of the Mathematical Sciences Regional Conference Series
 in Mathematics \textbf{88}, American Mathematical Society, Providence, 1996.

\bibitem{58}C. Huneke and R. Y. Sharp, {\it Bass numbers of local
 cohomology modules\/}, Transactions Amer.\ Math.\ Soc.\ {\bf 339} (1993)
 765--779.

\bibitem{MK}M. Katzman, {\it Parameter-test-ideals of Cohen--Macaulay
 rings\/}, Compositio Math.\ {\bf 144} (2008) 933--948.

\bibitem{KS}M. Katzman and R. Y. Sharp, {\it Uniform behaviour of the
 Frobenius closures of ideals generated by regular sequences\/},
 J. Algebra {\bf 295} (2006) 231--246.

\bibitem{Lyube97}G. Lyubeznik, {\it $F$-modules: applications to local
 cohomology and $D$-modules in characteristic $p > 0$\/}, J. reine angew.\
 Math.\ {\bf 491} (1997) 65--130.

\bibitem{LyuSmi01}G. Lyubeznik and K. E. Smith, {\it On the commutation of
 the test ideal with localization and completion\/}, Transactions Amer.\ Math.\
 Soc.\ {\bf 353} (2001) 3149--3180.

\bibitem{HM}H. Matsumura, {\it Commutative ring theory\/}, Cambridge Studies in
 Advanced Mathematics {\bf 8}, Cambridge University Press, 1986.

\bibitem{ga}
 R. Y. Sharp, {\it Graded annihilators of modules over the
 Frobenius skew polynomial ring, and tight closure\/},
 Transactions Amer.\ Math.\ Soc.\ {\bf 359} (2007) 4237--4258.

\bibitem{gatcti}
 R. Y. Sharp, {\it Graded annihilators and tight closure test ideals\/},
 J. Algebra {\bf 322} (2009) 3410--3426.

\bibitem{SN}R. Y. Sharp and N. Nossem, {\it Ideals in a perfect closure,
 linear growth of primary decompositions, and tight closure\/},
 Transactions Amer.\ Math.\ Soc.\ {\bf 356} (2004) 3687--3720.

\bibitem{SV}
 D. W. Sharpe and P. V\'amos, {\it Injective modules\/},
 Cambridge Tracts in Mathematics and Mathematical Physics {\bf 62},
 Cambridge University Press, 1972.

\bibitem{Smi97}K. E. Smith, {\it $F$-rational rings have rational
 singularities\/}, Amer.\ J. Math.\ {\bf 119} (1997) 159--180.

\end{thebibliography}
\end{document}